\documentclass[10pt]{article}
\usepackage{lmodern}
\usepackage{amsmath}
\usepackage[T1]{fontenc}
\usepackage[utf8]{inputenc}
\usepackage{authblk}
\usepackage{amsfonts}
\usepackage{graphicx}
\usepackage{rotating}
\usepackage{amssymb}
\usepackage[english]{babel}
\usepackage{color}
\usepackage{amsthm}
\usepackage{graphicx}
\usepackage{mathrsfs}
\usepackage{makecell}
\usepackage{microtype}
\usepackage{enumerate}
\usepackage{mathscinet}
\usepackage{array}
\usepackage{multirow}
\usepackage{booktabs}
\usepackage{enumerate}
\usepackage[cal=boondoxo,bb=ams]{mathalfa}
\usepackage{hyperref}
\hypersetup{hidelinks}
\newtheorem{theorem}{Theorem}[section]
\newtheorem{prop}{Proposition}[section]
\newtheorem{lemma}{Lemma}[section]
\newtheorem{coro}{Corollary}[section]
\newtheorem{remark}{Remark}[section]
\newtheorem{exam}{Example}[section]

\newcommand{\ml}{\mathcal}
\newcommand{\mb}{\mathbb}

\newcommand{\R}{\mathbb{R}}

\newcommand{\ity}{\infty}

\DeclareMathOperator{\non}{non}
\DeclareMathOperator{\lin}{lin}
\DeclareMathOperator{\diag}{diag}

\DeclareMathOperator{\intt}{int}
\DeclareMathOperator{\extt}{ext}
\DeclareMathOperator{\midd}{mid}

\title{On the Cauchy problem for semilinear regularity-loss-type $\sigma$-evolution models with memory term}
\author[1]{Wenhui Chen\thanks{Corresponding author: Wenhui Chen (wenhui.chen.math@gmail.com)}}
\author[1,2]{Tuan Anh Dao}
\affil[1]{Institute of Applied Analysis, Faculty of Mathematics and Computer Science, Technical University Bergakademie Freiberg, Pr\"uferstra{\ss}e 9, 09596 Freiberg, Germany}
\affil[2]{School of Applied Mathematics and Informatics, Hanoi University of Science and Technology, No.1 Dai Co Viet road, Hanoi, Vietnam}

\date{}
\begin{document}

\maketitle
\begin{abstract}
In this paper, we consider the Cauchy problem for semilinear $\sigma$-evolution models with an exponential decay memory term. Concerning the corresponding linear Cauchy problem, we derive some regularity-loss-type estimates of solutions and generalized diffusion phenomena. Particularly, the obtained estimations for solutions are sharper than those in the previous paper \cite{Liu-Ueda-2020}. Then, we determine the critical exponents for the semilinear Cauchy problem with power nonlinearity in some spatial dimensions by proving global (in time) existence of Sobolev solutions with low regularity of fractional orders and blow-up result for the Sobolev solutions even for any fractional value of $\sigma\geqslant 1$.\\
	
	
	\noindent\textbf{AMS Classification (2010)} Primary: 35R11; Secondary: 35B33, 35B40, 35L30, 35L76.
\end{abstract}
\fontsize{12}{15}
\selectfont

\section{Introduction}
This paper is to devote to the following Cauchy problem for the semilinear $\sigma$-evolution models with memory term:
\begin{align}\label{Semi_Generalized_Plate_Memory}
\begin{cases}
u_{tt}+(-\Delta)^{\sigma}u+u-g\ast u=|u|^p,&x\in\mb{R}^n,\ t>0,\\
u(0,x)=0,\ u_t(0,x)=u_1(x),&x\in\mb{R}^n,
\end{cases}
\end{align}
where $\sigma\geqslant 1$, $n\geqslant 1$ and $p>1$, the time-dependent function $g=g(t)$ stands for the exponential decay memory kernel $g(t):=\mathrm{e}^{-t}$. Here, the convolution term is denoted by
\begin{align*}
(g\ast u)(t,x):=\int_0^tg(t-\tau)u(\tau,x)\mathrm{d}\tau.
\end{align*}

In the last two decades, the Cauchy problem for the evolution equations with memory terms have caught a lot of attention from many mathematicians due to their wide applications in physics, mechanics and so on. For example, concerning $\sigma=1$ in \eqref{Semi_Generalized_Plate_Memory}, it corresponds to a model of ionized atmosphere \cite{Munoz-Naso-Vuk}. 

Let us begin with introducing the linear plate equation with memory
\begin{align}\label{Liu-Ueda-Plate}
\begin{cases}
u_{tt}+\Delta^2u+u-g\ast u=0,&x\in\mb{R}^n,\ t>0,\\
u(0,x)=u_0(x),\ u_t(0,x)=u_1(x),&x\in\mb{R}^n,
\end{cases}
\end{align} 
where the positive memory kernel $g\in\ml{C}^2([0,\infty))$ fulfilling $-C_1g(t)\leqslant g'(t)\leqslant -C_2g(t)$, $|g''(t)|\leqslant C_3g(t)$ and $1-\int_0^tg(\tau)\mathrm{d}\tau\geqslant0$ for any $t\geqslant0$ and some positive constants $C_1,C_2,C_3$. One may find that the above assumptions lead to 
\begin{align}\label{Liu-Ueda-Memory-kernel}
0<g(0)\,\mathrm{e}^{-C_1t}\leqslant g(t)\leqslant g(0)\,\mathrm{e}^{-C_2t}.
\end{align}
The authors of \cite{Liu-Ueda-2020} derived some estimates for solutions by using energy methods in the Fourier space. They found the decay structure of the regularity-loss-type, which plays an important role in some physical models including Timoshenko system, Euler-Maxwell system, thermoelastic plate equations and Moore-Gibson-Thompson equations (see, for example,  \cite{Ueda-Duan-Kawashima,Ueda-Racke2016,Pellicer-Said2019}). Roughly speaking, the term $u-g\ast u$ generates a weak dissipative effect. Furthermore, by applying spectral representation of solutions and asymptotic expansions of eigenvalues and their corresponding eigenprojections, the authors \cite{Liu-Ueda-2020} proved the sharpness of estimations for a special unknown variable (see (3.2) in \cite{Liu-Ueda-2020}), and found approximate equations to derive asymptotic profiles of solutions in a framework of weighted $L^1$ space. For the sake of the assumption \eqref{Liu-Ueda-Memory-kernel} and motivated by the model \eqref{Liu-Ueda-Plate}, it seems reasonable to consider an exponential decay kernel in generalized plate (or $\sigma$-evolution) models \eqref{Semi_Generalized_Plate_Memory} for any $\sigma\geqslant 1$. However, we cannot simply generalize the approach of \cite{Liu-Ueda-2020}) to this paper, which will be shown later. Let us now consider the semilinear plate equation with memory term, where the nonlinear term on the right-hand side depends on $u$ only, namely,
\begin{align}\label{Liu-Kawashima-Semilinear-Plate}
\begin{cases}
u_{tt}+\Delta^2u+u+g\ast\Delta u=f(u),&x\in\mb{R}^n,\ t>0,\\
u(0,x)=u_0(x),\ u_t(0,x)=u_1(x),&x\in\mb{R}^n.
\end{cases}
\end{align} 
In the recent paper \cite{Liu-Kawashima-2011}, the authors proved global (in time) existence of small data Sobolev solutions with high regularity such that
\begin{align*}
u\in\ml{C}\left([0,\infty),H^{s+1}(\mb{R}^n)\right)
\end{align*}
by using embedding theory and contraction mapping theorem, where the regularity-loss effect appears due to the corresponding linear model of \eqref{Liu-Kawashima-Semilinear-Plate}. Here, $s$ is an integer number satisfying $s\geqslant2$ for $n=1$, and $s\geqslant\lceil n/2\rceil$ for $n\geqslant2$. It seems interesting to develop global (in time) existence results with lower regularity of fractional orders. Concerning other studies on the Cauchy problem for plate equations with memory terms, we refer to \cite{Sugitani-Kawashima-2010,Liu-Kawashima-2011-2,Liu-2012,Mao-Liu-2014,Mao-Liu-2017} and the references therein.

Let us turn to the $\sigma$-evolution equations with friction $u_t$ as follows:
\begin{align}\label{Duong-Rei-friction}
\begin{cases}
u_{tt}+(-\Delta)^{\sigma}u+u_t=|u|^p,&x\in\mb{R}^n,\ t>0,\\
u(0,x)=u_0(x),\ u_t(0,x)=u_1(x),&x\in\mb{R}^n.
\end{cases}
\end{align}
When $\sigma$ is an integer number, \eqref{Duong-Rei-friction} is the so-called polyharmonic damped wave equation which was explored in \cite{Takeda}. The author has found a critical condition to catch the global (in time) existence of small data solutions to \eqref{Duong-Rei-friction} in the space dimension $n=1$. One recognizes that the restriction of the one-dimensional case causes by some technical difficulties. Involving the general case of any fractional number $\sigma \geqslant 1$, the authors in the recent paper \cite{Duong-Reissig-2017} applied some of sharp decay estimates for solutions, which were well-studied in \cite{RaduTodorovaYordanov}, to get the possible range of the admissible exponents of power nonlinearity for \eqref{Duong-Rei-friction}. The point worth noticing is that by proving a diffusion phenomenon in the abstract setting, the authors in \cite{RaduTodorovaYordanov} have achieved these sharp decay estimates thanks to the application of the Nash inequality and the Markov property for the parabolic semigroup. However, we should remark that it is challenging to follow the previous studies \cite{Duong-Reissig-2017,RaduTodorovaYordanov,Takeda} in the way of directly computing the eigenvalues and representations of solutions since the memory term in \eqref{Semi_Generalized_Plate_Memory} brings difficulties. Moreover, we will observe later in Remark \ref{Compare-Duong-Reissig} that friction $u_t$ has a stronger damping effect than that generated by the memory term $u-g\ast u$, so how does this memory-type damping influence on the critical exponent for \eqref{Semi_Generalized_Plate_Memory} is still unknown. We will give a positive answer of this question in the present paper. For other studies on $\sigma$-evolution equations with damping term, we refer the interested readers to \cite{D'Ab-Ebert-2017,DaoReissig1,DaoReissig2,D'Ab-Gir-Lia,D'Ab-Gir-2020}.

Our main purpose of this paper is to investigate some qualitative properties of solutions to the corresponding linear and the semilinear $\sigma$-evolution models \eqref{Semi_Generalized_Plate_Memory} with an exponential decay memory term. Particularly, we are interested in the influence of a weak dissipation generated by the term containing $u-g\ast u$ on the $\sigma$-evolution models for any $\sigma\geqslant1$.

Concerning the linearized Cauchy problem for \eqref{Semi_Generalized_Plate_Memory}, we prove some sharper estimates for solutions based on the $L^2$ norm, which are of the regularity-loss-type. In additional, we find two different approximate equations to the linear problem (see \eqref{Ref_Eq_Small} and \eqref{Ref_Eq_Large} later). Nevertheless, the studies of these properties, especially estimates for solutions, are not a simple generalization of those in previous studies. For one thing, due to the non-local operator $(-\Delta)^{\sigma}$ with $\sigma\geqslant1$ in \eqref{Linear_Generalized_Plate_Memory} for all dimensions, it seems that the approach, i.e. spectral theory associated with asymptotic expansions, in \cite{Liu-Ueda-2020} does not work well in our model, and WKB analysis associated with diagonalization procedures is effective to deal with the fractional power operator. For another, by some calculations, we derive the estimates for solutions and their time-derivative in the Fourier space independently, which provides us an effective way to get sharper estimates for solutions than those shown in \cite{Liu-Ueda-2020}.

Regarding the semilinear Cauchy problem \eqref{Semi_Generalized_Plate_Memory}, the critical exponents
\begin{align*}
p=p_{\mathrm{crit}}(n,m,\sigma)
\end{align*}
by assuming $L^m$ regular data with $m\in[1,2)$ for some values of $n,\sigma$ will be found. On the one hand, global (in time) small data Sobolev solutions with lower regularity in the supercritical case $p>p_{\mathrm{crit}}(n,m,\sigma)$ for some values of $n,m,\sigma$ will be shown by constructing time-weighted Sobolev space and employing the fractional Gagliardo-Nirenberg inequality combined with the fractional Leibniz rule and the fractional chain rule. Up to now, it seems that we do not have any previous research manuscripts to study lower regular global (in time) Sobolev solutions in the framework of regularity-loss. One the other hand, in the subcritical case $1<p<p_{\mathrm{crit}}(n,m,\sigma)$ for all $n\geqslant 1$, $m\in[1,2)$ and $\sigma\geqslant1$, we will prove that every local (in time) Sobolev solution blows up by a modified test function method to deal with the fractional power operator $(-\Delta)^{\sigma}$, which is motivated by the quite recent paper \cite{Dao-Reissig-2020}. To deal with the memory effect, we construct two different functionals related to the solution itself. Especially, when $m=1$ and $\sigma$ is assumed to be an integer number, every weak solution will blows up in the critical case $p=p_{\mathrm{crit}}(n,1,\sigma)$. 

The remaining part of this paper is organized as follows. We prepare some sharp regularity-loss-type estimates for solutions to the corresponding linear model of \eqref{Semi_Generalized_Plate_Memory} by using WKB analysis and multi-step diagonalization procedures in Section \ref{Section_Linear_Problem}. Additionally, the generalized diffusion phenomena also will be shown in the end of Section \ref{Section_Linear_Problem}. Then, by using the derived estimates and some tools in Harmonic Analysis, in Section \ref{Section_GESDS} we prove global (in time) existence of small data Sobolev solutions with lower regularity. To conclude the optimility, we employ a modified test function method and construct a contradiction argument to prove blow-up result for the Sobolev solutions in Section \ref{Section_Blowup}. Eventually, final remarks in Section \ref{Section_Finalremark} are to complete this paper.

\medskip

\noindent\textbf{Notations:} In this paper, $f\lesssim g$ means that there exists a positive constant $C$ such that $f\leqslant Cg$. Also, we denote $\lceil r\rceil:=\min\{C\in\mb{Z}: r\leqslant C\}$ and $[r]$ as the ceiling function and the integer part of $r \in \mb{R}$, respectively. We define the so-called Japanese bracket $\langle x\rangle:=(1+|x|^2)^{1/2}$ for all $x\in\mb{R}^n$.  Moreover, $\dot{H}^s_q(\mb{R}^n)$ with $s\geqslant0$ and $1\leqslant q<\infty$, denote Riesz potential spaces based on $L^q(\mb{R}^n)$. Furthermore, $|D|^s$ with $s\geqslant0$ stand for the pseudo-differential operators with symbol $|\xi|^s$. We denote the identity matrix of dimensions $k\times k$ by $I_{k}$. Finally, let us define the following zones for the phase space:
\begin{align*}
Z_{\intt}(\varepsilon)&:=\left\{\xi\in\mb{R}^n:|\xi|<\varepsilon\ll1\right\},\\
Z_{\midd}(\varepsilon,N)&:=\left\{\xi\in\mb{R}^n:\varepsilon\leqslant |\xi|\leqslant N\right\},\\
Z_{\extt}(N)&:=\left\{\xi\in\mb{R}^n:|\xi|> N\gg1\right\}.
\end{align*}
The cut-off functions $\chi_{\intt}(\xi),\chi_{\midd}(\xi),\chi_{\extt}(\xi)\in \mathcal{C}^{\infty}(\mb{R}^n)$ having their supports in the zone $Z_{\intt}(\varepsilon)$, $Z_{\midd}(\varepsilon/2,2N)$ and $Z_{\extt}(N)$, respectively, so that $\chi_{\intt}(\xi)+\chi_{\midd}(\xi)+\chi_{\extt}(\xi)=1$ for all $\xi \in \mb{R}^n$.

\section{Qualitative properties of solutions to the linear Cauchy problem}\label{Section_Linear_Problem}
\setcounter{equation}{0}
Our starting point in this paper is to study the corresponding linear Cauchy problem for \eqref{Semi_Generalized_Plate_Memory}, namely,
\begin{align}\label{Linear_Generalized_Plate_Memory}
\begin{cases}
u_{tt}+(-\Delta)^{\sigma}u+u-g\ast u=0,&x\in\mb{R}^n,\ t>0,\\
u(0,x)=0,\ u_t(0,x)=u_1(x),&x\in\mb{R}^n,
\end{cases}
\end{align}
where $\sigma\geqslant 1$, $n\geqslant 1$ and $g(t)=\mathrm{e}^{-t}$.

In this section, we will derive some qualitative properties of solutions to \eqref{Linear_Generalized_Plate_Memory}, which will be used later. Precisely, some estimates for solutions in the $\dot{H}^s$  norm with $s\geqslant0$ and generalized diffusion phenomena are of our interest. Again, the approach, i.e. spectral theory associated with asymptotic expansions of eigenvalues and their corresponding eigenprojections, does not work well in our model \eqref{Linear_Generalized_Plate_Memory} due to the non-local operator $(-\Delta)^{\sigma}$ with any $\sigma\geqslant 1$. Therefore, to overcome the difficulty for the fractional case, motivated by some recent papers studying non-local operators in coupled systems (see, for instance, \cite{Reissig-2016,Chen-Reissig-2019,Chen-CMS,Chen-JMAA}), we may employ WKB analysis and multi-step diagonalization procedures (see, for example, \cite{Yagdjian-1997,Reissig-Wang-2005,Jachmann-2008}).

Before applying these techniques, we will reduce \eqref{Linear_Generalized_Plate_Memory} to a suitable evolution system, which is the first-order system with respect to time variable. Observing from the exponential memory kernel having the property
\begin{align}\label{Eq_derivative_g*u}
\frac{\partial}{\partial t}(g\ast u)(t,x)=u(t,x)-(g\ast u)(t,x),
\end{align}
we are motivated to define a suitable ansatz $U=U(t,x)$ to \eqref{Linear_Generalized_Plate_Memory} such that
\begin{align}\label{Formu_U}
U(t,x):=\left(u_t(t,x)+i(-\Delta)^{\sigma/2}u(t,x),u_t(t,x)-i(-\Delta)^{\sigma/2}u(t,x),(g\ast u)(t,x)-u(t,x)\right)^{\mathrm{T}}
\end{align}
carrying its corresponding initial data
\begin{align}\label{Formu_U0}
U_0(x):=U(0,x)=(u_1(x),u_1(x),0)^{\mathrm{T}}.
\end{align}
By using the differential property \eqref{Eq_derivative_g*u}, one may reduce \eqref{Linear_Generalized_Plate_Memory} to
\begin{align}\label{Eq_First_Order_x}
\begin{cases}
U_t-A_0(-\Delta)^{\sigma/2}U-A_1U=0,&x\in\mb{R}^n,\ t>0,\\
U(0,x)=U_0(x),&x\in\mb{R}^n,
\end{cases}
\end{align}
where the coefficient matrices are given by
\begin{align*}
A_0=\left(
{\begin{array}{*{20}c}
	i & 0  &  0 \\
	0  & -i &  0\\
	0 & 0  & 0 \\
	\end{array}}
\right)\ \ \mbox{and}\ \ A_1=\left(
{\begin{array}{*{20}c}
	0 & 0  & 1 \\
	0  & 0 &  1 \\
	-\frac{1}{2} & -\frac{1}{2}  & -1 \\
	\end{array}}
\right).
\end{align*}
Let us apply the partial Fourier transform $\hat{U}(t,\xi)=\ml{F}_{x\to \xi}(U(t,x))$ in \eqref{Eq_First_Order_x}. Then, the Cauchy problem for the first-order $|\xi|$-dependent system reads as follows:
\begin{align}\label{Eq_First_Order_xi}
\begin{cases}
\hat{U}_t-A_0|\xi|^{\sigma}\hat{U}-A_1\hat{U}=0,&\xi\in\mb{R}^n,\ t>0,\\
\hat{U}(0,\xi)=\hat{U}_0(\xi),&\xi\in\mb{R}^n.
\end{cases}
\end{align}
With the aim of understanding the influence of the parameter $|\xi|$, we will discuss the asymptotic behavior of solutions into three parts containing small frequencies $\xi\in Z_{\intt}(\varepsilon)$, large frequencies $\xi\in Z_{\extt}(N)$ and middle frequencies $\xi\in Z_{\midd}(\varepsilon,N)$, individually. 

\subsection{Representation of solutions} \label{Rep.Sol}
In this part, we will study representations of solutions to \eqref{Eq_First_Order_xi}. Particularly, we can apply diagonalization procedure to derive representation of solutions for $\xi\in Z_{\intt}(\varepsilon)\cup Z_{\extt}(N)$, and establish a contradiction argument to get an exponential stability for $\xi\in Z_{\midd}(\varepsilon,N)$. Actually, we found that the Cauchy problem \eqref{Eq_First_Order_x} is a special case of the so-called $\alpha-\beta$ system, which can be shown by
\begin{align*}
\begin{cases}
u_{tt}+\ml{A}u-\gamma_1\ml{A}^{\alpha}v=0,&x\in\mb{R}^n,\ t>0,\\
v_t+\gamma_2\ml{A}^{\beta}v+\gamma_1\ml{A}^{\alpha}u_t=0,&x\in\mb{R}^n,\ t>0,\\
u(0,x)=u_0(x),\ u_t(0,x)=u_1(x),\ v(0,x)=v_0(x),&x\in\mb{R}^n,
\end{cases}
\end{align*}
if we use a suitable change of variable (see (2.1) in \cite{Liu-Chen-2020}) and take $\alpha=\beta=0$, $\gamma_1=\gamma_2=1$, $\ml{A}=(-\Delta)^{\sigma}$ with $\sigma\geqslant 1$. Therefore, we may directly apply the results studied in \cite{Liu-Chen-2020}. More precisely, the following three propositions are derived in Propositions 2.1, 2.2, 2.4 in \cite{Liu-Chen-2020} with $\alpha=0$, respectively.
\begin{prop}\label{Prop_Repre_Small}
	The eigenvalues $\lambda_j=\lambda_j(|\xi|)$ for $j=1,2,3$, of the coefficient matrix $A_0|\xi|^{\sigma}+A_1$ from \eqref{Eq_First_Order_xi} behave for $|\xi|<\varepsilon\ll 1$ as follows:
	\begin{align*}
	\lambda_1(|\xi|)&=-|\xi|^{2\sigma}+\ml{O}(|\xi|^{3\sigma}),\\
	\lambda_2(|\xi|)&=-\left(\tfrac{1}{2}+i\tfrac{\sqrt{3}}{2}\right)+\left(\tfrac{1}{2}-i\tfrac{\sqrt{3}}{6}\right)|\xi|^{2\sigma}+\ml{O}(|\xi|^{3\sigma}),\\
	\lambda_3(|\xi|)&=-\left(\tfrac{1}{2}-i\tfrac{\sqrt{3}}{2}\right)+\left(\tfrac{1}{2}+i\tfrac{\sqrt{3}}{6}\right)|\xi|^{2\sigma}+\ml{O}(|\xi|^{3\sigma}).
	\end{align*}
	Furthermore, the solution to \eqref{Eq_First_Order_xi} has for $\xi\in Z_{\intt}(\varepsilon)$ the representation
	\begin{align*}
	\hat{U}(t,\xi)=T_{\intt}\diag\left(\mathrm{e}^{\lambda_1(|\xi|)t},\mathrm{e}^{\lambda_2(|\xi|)t},\mathrm{e}^{\lambda_3(|\xi|)t}\right)T_{\intt}^{-1}\hat{U}_0(\xi),
	\end{align*}
	where $T_{\intt}:=N_1(I_{3}+N_2(|\xi|))(I_3+N_3(|\xi|))$ and
	\begin{align*}
	N_1:=\left(
	{\begin{array}{*{20}c}
		-1 & \frac{-1+i\sqrt{3}}{2} & -\frac{1+i\sqrt{3}}{2}\\
		1 & \frac{-1+i\sqrt{3}}{2} & -\frac{1+i\sqrt{3}}{2}\\
		0 & 1 & 1
		\end{array}}
	\right)&,\ \ N_2(|\xi|):=|\xi|^{\sigma}\left(
	{\begin{array}{*{20}c}
		0 & -\frac{\sqrt{3}+i}{1+i\sqrt{3}} & \frac{-\sqrt{3}+i}{-1+i\sqrt{3}}\\
		-\frac{2\sqrt{3}}{3(1+i\sqrt{3})} & 0 & 0\\
		\frac{2\sqrt{3}}{3(1-i\sqrt{3})} & 0 & 0
		\end{array}}
	\right),\\
	N_3(|\xi|):&=|\xi|^{2\sigma}\left(
	{\begin{array}{*{20}c}
		0 & 0 & 0\\
		0 & 0 & \frac{-1+i\sqrt{3}}{6}\\
		0 & -\frac{1+i\sqrt{3}}{6} & 0
		\end{array}}
	\right).
	\end{align*} 
\end{prop}
\begin{prop}\label{Prop_Repre_Large}
	The eigenvalues $\lambda_j=\lambda_j(|\xi|)$ for $j=1,2,3$, of the coefficient matrix $A_0|\xi|^{\sigma}+A_1$ from \eqref{Eq_First_Order_xi} behave for $|\xi|>N\gg 1$ as follows:
	\begin{align*}
	\lambda_1(|\xi|)&=i|\xi|^{\sigma}+\tfrac{i}{2}|\xi|^{-\sigma}-\tfrac{1}{2}|\xi|^{-2\sigma}+\ml{O}(|\xi|^{-3\sigma}),\\
	\lambda_2(|\xi|)&=-i|\xi|^{\sigma}-\tfrac{i}{2}|\xi|^{-\sigma}-\tfrac{1}{2}|\xi|^{-2\sigma}+\ml{O}(|\xi|^{-3\sigma}),\\
	\lambda_3(|\xi|)&=-1+|\xi|^{-2\sigma}+\ml{O}(|\xi|^{-3\sigma}).
	\end{align*}
	Furthermore, the solution to \eqref{Eq_First_Order_xi} has for $\xi\in Z_{\extt}(N)$ the representation
	\begin{align*}
	\hat{U}(t,\xi)=T_{\extt}\diag\left(\mathrm{e}^{\lambda_1(|\xi|)t},\mathrm{e}^{\lambda_2(|\xi|)t},\mathrm{e}^{\lambda_3(|\xi|)t}\right)T_{\extt}^{-1}\hat{U}_0(\xi),
	\end{align*}
	where $T_{\extt}:=(I_{3}+N_4(|\xi|))(I_{3}+N_5(|\xi|))(I_{3}+N_6(|\xi|))$ and
	\begin{align*}
	N_4(|\xi|):=|\xi|^{-\sigma}\left(
	{\begin{array}{*{20}c}
		0 & 0 & i\\
		0 & 0 & -i\\
		\frac{i}{2} & -\frac{i}{2} & 0
		\end{array}}
	\right)&,\ \ N_5(|\xi|):=|\xi|^{-2\sigma}\left(
	{\begin{array}{*{20}c}
		0 & \frac{1}{4} & -1\\
		\frac{1}{4} & 0 & -1\\
		-\frac{1}{2} & -\frac{1}{2} & 0
		\end{array}}
	\right),\\
	N_6(|\xi|):=|\xi|^{-3\sigma}&\left(
	{\begin{array}{*{20}c}
		0 & \frac{i}{4} & -i\\
		-\frac{i}{4} & 0 & i\\
		-\frac{i}{2} & \frac{i}{2} & 0
		\end{array}}
	\right).
	\end{align*}
\end{prop}

\begin{prop}\label{Prop_Middle_Freq}
	The solution $\hat{U}=\hat{U}(t,\xi)$ to \eqref{Eq_First_Order_xi} fulfills the next estimate for any $\xi\in Z_{\midd}(\varepsilon,N)$:
	\begin{align*}
	|\hat{U}(t,\xi)|\lesssim \mathrm{e}^{-ct}|\hat{U}_0(\xi)|
	\end{align*}
	with a positive constant $c$ and for any $t \geqslant 0$.
\end{prop}

\subsection{Regularity-loss-type estimates of solutions}
The pointwise estimates in the Fourier space is to depict the decay properties of solutions. In Section \ref{Rep.Sol}, the asymptotic behavior of eigenvalues and representation of solutions $\hat{U}(t,\xi)$ have been derived in Propositions \ref{Prop_Repre_Small}, \ref{Prop_Repre_Large}, \ref{Prop_Middle_Freq}. We may now summarize them to get the sharp pointwise estimates in the Fourier space.
\begin{prop}\label{Prop_Pointwise}
	The solution $\hat{U}=\hat{U}(t,\xi)$ to \eqref{Eq_First_Order_xi} fulfills the pointwise estimates for any $\xi\in\mb{R}^n$ and $t\geqslant 0$ as follows:
	\begin{align*}
	|\hat{U}(t,\xi)|\lesssim\exp\left(-c\frac{|\xi|^{2\sigma}}{(1+|\xi|^2)^{2\sigma}}t\right)|\hat{U}_0(\xi)|
	\end{align*}
	with a positive constant $c>0$.
\end{prop}

Let us recall a useful lemma which was proved in \cite{Liu-Chen-2020}.
\begin{lemma}\label{Lemma.2} \cite[Lemma~3.2]{Liu-Chen-2020}
	Let us consider $f\in H^{s+\ell}(\mb{R}^n)$ with $s\geqslant0$ and $\ell\geqslant0$. Then, the following estimate holds:
	\begin{align*}
	\left\|\chi_{\extt}(D)|D|^{s}\ml{F}_{\xi\rightarrow x}^{-1}\left(\mathrm{e}^{-c|\xi|^{-\theta}t}\right)(t,x)f(x)\right\|_{L^2(\mb{R}^n)}\lesssim
	(1+t)^{-\frac{\ell}{\theta}}\|f\|_{H^{s+\ell}(\mb{R}^n)}
	\end{align*}
	where $\theta>0$ and $c>0$.
\end{lemma}

By using Proposition \ref{Prop_Pointwise} and Lemma \ref{Lemma.2}, we may conclude the next energy estimates with initial data taken from Sobolev spaces with additional $L^m$ regularities for any $m\in[1,2]$.
\begin{theorem}\label{Thm_Energy_U}
	Let us assume $U_0\in (H^{s+\ell}(\mb{R}^n)\cap L^{m}(\mb{R}^n))^3$ with $s\geqslant0$, $\ell\geqslant0$ and $m\in[1,2]$. Then, the solution $U=U(t,x)$ to \eqref{Eq_First_Order_x} with $\sigma\geqslant1$ fulfills the estimates as follows:
	\begin{align*}
	\|U(t,\cdot)\|_{(\dot{H}^{s}(\mb{R}^n))^3}\lesssim
	(1+t)^{-\frac{n}{2\sigma}(\frac{1}{m}-\frac{1}{2})-\frac{s}{2\sigma}}\|U_0\|_{(L^m(\mb{R}^n))^3}+
	(1+t)^{-\frac{\ell}{2\sigma}}\|U_0\|_{(H^{s+\ell}(\mb{R}^n))^3}
	\end{align*}
	for any $t\geqslant0$.
\end{theorem}
\begin{proof}
	Let us begin with the case when $m\in[1,2)$. From the pointwise estimates stated in Proposition \ref{Prop_Pointwise}, we deduce
	\begin{align*}
	\left\|\chi_{\intt}(D)|D|^{s}U(t,\cdot)\right\|_{(L^2(\mb{R}^n))^3}&\lesssim\left\|\chi_{\intt}(\xi)|\xi|^{s}\mathrm{e}^{-c|\xi|^{2\sigma}t}\hat{U}_0(\xi)\right\|_{(L^2(\mb{R}^n))^3}\\
	&\lesssim\left(\int_0^{\varepsilon}r^{sm_1+n-1}\mathrm{e}^{-cm_1r^{2\sigma}t}\mathrm{d}r\right)^{1/m_1}\|U_0\|_{(L^m(\mb{R}^n))^3}\\
	&\lesssim (1+t)^{-\frac{n}{2\sigma}(\frac{1}{m}-\frac{1}{2})-\frac{s}{2\sigma}}\|U_0\|_{(L^m(\mb{R}^n))^3},
	\end{align*}
	where $1/m_1=(2-m)/(2m)$ and we used H\"older's inequality as well as the Hausdorff-Young inequality. For the case of large frequencies, we apply Lemma \ref{Lemma.2} to get
	\begin{align*}
	\left\|\chi_{\extt}(D)|D|^{s}U(t,\cdot)\right\|_{(L^2(\mb{R}^n))^3}&\lesssim\left\|\chi_{\extt}(\xi)|\xi|^{s}\mathrm{e}^{-c|\xi|^{-2\sigma}t}\hat{U}_0(\xi)\right\|_{(L^2(\mb{R}^n))^3}\\
	&\lesssim(1+t)^{-\frac{\ell}{2\sigma}}\|U_0\|_{(H^{s+\ell}(\mb{R}^n))^3}.
	\end{align*}
	Finally, due to the exponential stability when $\xi\in Z_{\midd}(\varepsilon,N)$, we can complete the statement by summarizing the derived estimates in each zone.
	
	Let us consider the limit case when $m=2$. Thus, we may now apply the norm inequality $\|\cdot\|_{L^2(\mb{R}^n)}\lesssim\|\cdot\|_{L^{\infty}(\mb{R}^n)}\|\cdot\|_{L^2(\mb{R}^n)}$ for the estimates in each zone. By similar calculations to the case $m \in [1,2)$, we may complete the proof. 
\end{proof}

To investigate the global (in time) existence of solutions in the forthcoming part, the estimates for solutions and their derivatives play an important role. However, providing that we directly apply the estimates for $\hat{U}(t,\xi)$ to get the pointwise estimate for $\hat{u}(t,\xi)$, $|\xi|^{\sigma}\hat{u}(t,\xi)$ and $\hat{u}_t(t,\xi)$, then some loss of decay rates will appear. To overcome this difficulty, we may explicitly calculate the principal parts of representation of solutions from Propositions \ref{Prop_Repre_Small} and \ref{Prop_Repre_Large}. Combining
\begin{align*}
|\xi|^{\sigma}\hat{u}(t,\xi)=\frac{1}{2i}(1,-1,0)\, \hat{U}(t,\xi)\ \ \mbox{and}\ \ \hat{u}_t(t,\xi)=\frac{1}{2}(1,1,0)\, \hat{U}(t,\xi),
\end{align*}
we can derive the following sharper pointwise estimates.

\begin{theorem}\label{Thm_Est_Solution_Derivative}
	Let us assume $u_1\in H^{s+\ell}(\mb{R}^n)\cap L^m(\mb{R}^n)$ with $s\geqslant0$, $\ell\geqslant0$ and $m\in[1,2]$. Then, the solution $u=u(t,x)$ to \eqref{Linear_Generalized_Plate_Memory} fulfills the estimates as follows:
	\begin{align*}
	\|u(t,\cdot)\|_{L^2(\mb{R}^n)}&\lesssim (1+t)^{-\frac{n}{2\sigma}(\frac{1}{m}-\frac{1}{2})} \|u_1\|_{ L^m(\mb{R}^n)}+(1+t)^{-\frac{1}{2}-\frac{\ell}{2\sigma}}\|u_1\|_{H^{\ell}(\mb{R}^n)},\\
	\|\,|D|^{\sigma}u(t,\cdot)\|_{\dot{H}^{s}(\mb{R}^n)}&\lesssim (1+t)^{-\frac{n}{2\sigma}(\frac{1}{m}-\frac{1}{2})-\frac{1}{2}-\frac{s}{2\sigma}}\|u_1\|_{L^m(\mb{R}^n)}+(1+t)^{-\frac{\ell}{2\sigma}}\|u_1\|_{H^{s+\ell}(\mb{R}^n)},\\
	\|u_t(t,\cdot)\|_{\dot{H}^{s}(\mb{R}^n)}&\lesssim (1+t)^{-\frac{n}{2\sigma}(\frac{1}{m}-\frac{1}{2})-1-\frac{s}{2\sigma}}\|u_1\|_{L^m(\mb{R}^n)}+(1+t)^{-\frac{\ell}{2\sigma}}\|u_1\|_{H^{s+\ell}(\mb{R}^n)},
	\end{align*}
	for any $t\geqslant0$. 
\end{theorem}
\begin{proof}
	Let us consider the solutions localized to small frequency zone firstly. According to Proposition \ref{Prop_Repre_Small}, we may compute explicit formula of coefficient matrix $T_{\intt}$ and its inverse. By taking into consideration $\xi\in Z_{\intt}(\varepsilon)$, it yields
	\begin{align*}
	\chi_{\intt}(\xi)|\xi|^{\sigma}|\hat{u}(t,\xi)|&\lesssim\chi_{\intt}(\xi)\left(\mathrm{e}^{\lambda_1(|\xi|)t}+\tfrac{\sqrt{3}i}{6}\left|\mathrm{e}^{\lambda_2(|\xi|)t}-\mathrm{e}^{\lambda_3(|\xi|)t}\right|+\tfrac{1}{2}\left|\mathrm{e}^{\lambda_2(|\xi|)t}+\mathrm{e}^{\lambda_3(|\xi|)t}\right|\right)|\xi|^{\sigma}|\hat{u}_1(\xi)|,\\
	\chi_{\intt}(\xi)|\hat{u}_t(t,\xi)|&\lesssim\chi_{\intt}(\xi)\left(\mathrm{e}^{\lambda_1(|\xi|)t}+\tfrac{\sqrt{3}i}{9}\left|\mathrm{e}^{\lambda_2(|\xi|)t}-\mathrm{e}^{\lambda_3(|\xi|)t}\right|+\tfrac{1}{3}\left|\mathrm{e}^{\lambda_2(|\xi|)t}+\mathrm{e}^{\lambda_3(|\xi|)t}\right|\right)|\xi|^{2\sigma}|\hat{u}_1(\xi)|.
	\end{align*}
	It is clear from Euler's formula that
	\begin{align*}
	\chi_{\intt}(\xi)\left|\mathrm{e}^{\lambda_2(|\xi|)t}-\mathrm{e}^{\lambda_3(|\xi|)t}\right|&\lesssim\chi_{\intt}(\xi)\, \mathrm{e}^{-\frac{t}{2}+|\xi|^{\sigma}t}\left|i\sin\left(\left(\tfrac{\sqrt{3}}{2}+\tfrac{\sqrt{3}}{6}|\xi|^{2\sigma}\right)t\right)\right|,\\
	\chi_{\intt}(\xi)\left|\mathrm{e}^{\lambda_2(|\xi|)t}+\mathrm{e}^{\lambda_3(|\xi|)t}\right|&\lesssim\chi_{\intt}(\xi)\, \mathrm{e}^{-\frac{t}{2}+|\xi|^{\sigma}t}\left|\cos\left(\left(\tfrac{\sqrt{3}}{2}+\tfrac{\sqrt{3}}{6}|\xi|^{2\sigma}\right)t\right)\right|.
	\end{align*}
	In other words, by using bounded estimates for the sine and cosine functions, one may have the following estimates for any $s\geqslant0$:
	\begin{align*}
	\chi_{\intt}(\xi)|\hat{u}(t,\xi)|&\lesssim \chi_{\intt}(\xi)\,\mathrm{e}^{-c_1|\xi|^{2\sigma}t}|\hat{u}_1(\xi)|,\\
	\chi_{\intt}(\xi)|\xi|^{s+\sigma}|\hat{u}(t,\xi)|&\lesssim \chi_{\intt}(\xi)\,|\xi|^{s+\sigma}\mathrm{e}^{-c_1|\xi|^{2\sigma}t}|\hat{u}_1(\xi)|,\\
	\chi_{\intt}(\xi)|\xi|^{s}|\hat{u}_t(t,\xi)|&\lesssim \chi_{\intt}(\xi)\,|\xi|^{s+2\sigma}\mathrm{e}^{-c_1|\xi|^{2\sigma}t}|\hat{u}_1(\xi)|,
	\end{align*}
	where $c_1$ is a suitable positive constant. By following the same approach as in the proof of Theorem \ref{Thm_Energy_U} for small frequencies, we can immediately estimate the solutions and their derivatives for initial data taken from $L^m$ spaces with $m\in[1,2]$.
	
	Next, we consider the solutions localized to large frequency zone. By the similar way to small frequency zone, from Proposition \ref{Prop_Repre_Large}, we may carry out some calculations to get
	\begin{align*}
	\chi_{\extt}(\xi)|\xi|^{\sigma}|\hat{u}(t,\xi)|&\lesssim \chi_{\extt}(\xi)\left|i\left(\mathrm{e}^{\lambda_1(|\xi|)t}-\mathrm{e}^{\lambda_2(|\xi|)t}\right)\right||\hat{u}_1(\xi)|+\chi_{\extt}(\xi)|\xi|^{-3\sigma}\mathrm{e}^{\lambda_3(|\xi|)t}|\hat{u}_1(\xi)|,\\
	\chi_{\extt}(\xi)|\hat{u}_t(t,\xi)|&\lesssim \chi_{\extt}(\xi)\left|\mathrm{e}^{\lambda_1(|\xi|)t}+\mathrm{e}^{\lambda_2(|\xi|)t}\right||\hat{u}_1(\xi)|+\chi_{\extt}(\xi)|\xi|^{-4\sigma}\mathrm{e}^{\lambda_3(|\xi|)t}|\hat{u}_1(\xi)|.
	\end{align*}
	Moreover, we know
	\begin{align*}
	\chi_{\extt}(\xi)\left|\mathrm{e}^{\lambda_1(|\xi|)t}-\mathrm{e}^{\lambda_2(|\xi|)t}\right|&\lesssim\chi_{\extt}(\xi)\, \mathrm{e}^{-\frac{1}{2}|\xi|^{-2\sigma}t}\left|\sin\left(\left(|\xi|^{\sigma}+\tfrac{1}{2}|\xi|^{-\sigma}\right)t\right)\right|,\\
	\chi_{\extt}(\xi)\left|\mathrm{e}^{\lambda_1(|\xi|)t}+\mathrm{e}^{\lambda_2(|\xi|)t}\right|&\lesssim\chi_{\extt}(\xi)\, \mathrm{e}^{-\frac{1}{2}|\xi|^{-2\sigma}t}\left|\cos\left(\left(|\xi|^{\sigma}+\tfrac{1}{2}|\xi|^{-\sigma}\right)t\right)\right|,
	\end{align*}
	and
	\begin{align*}
	\chi_{\extt}(\xi)|\xi|^{-j\sigma}\mathrm{e}^{\lambda_3(|\xi|)t}\lesssim\chi_{\extt}(\xi)\mathrm{e}^{-\tilde{c}t},
	\end{align*}
	with $j=3,4,$ and a positive constant $\tilde{c}$, which lead to the next estimates for any $s\geqslant0$:
	\begin{align*}
	\chi_{\extt}(\xi)|\hat{u}(t,\xi)|&\lesssim \chi_{\extt}(\xi)\,|\xi|^{-\sigma}\mathrm{e}^{-c_2|\xi|^{-2\sigma}t}|\hat{u}_1(\xi)|,\\
	\chi_{\extt}(\xi)|\xi|^{s+\sigma}|\hat{u}(t,\xi)|&\lesssim \chi_{\extt}(\xi)\,|\xi|^{s}\mathrm{e}^{-c_2|\xi|^{-2\sigma}t}|\hat{u}_1(\xi)|,\\
	\chi_{\extt}(\xi)|\xi|^{s}|\hat{u}_t(t,\xi)|&\lesssim \chi_{\extt}(\xi)\,|\xi|^{s}\mathrm{e}^{-c_2|\xi|^{-2\sigma}t}|\hat{u}_1(\xi)|,
	\end{align*}
	where $c_2$ is a suitable positive constant. By repeating the same way as in the proof of Theorem \ref{Thm_Energy_U} for large frequencies, the solutions and their derivatives can be estimated.
	
	Finally, due to the exponential stability for middle frequencies such that
	\begin{align*}
	\chi_{\midd}(\xi)|\xi|^s\left(|\hat{u}(t,\xi)|+|\xi|^{\sigma}|\hat{u}(t,\xi)|+|\hat{u}_t(t,\xi)|\right)\lesssim \chi_{\midd}(\xi) \mathrm{e}^{-ct}|\hat{u}_1(\xi)|
	\end{align*}
	with a positive constant $c>0$, we may have
	\begin{align*}
	\|\chi_{\midd}(D)u(t,\cdot)\|_{L^2(\mb{R}^n)}+\|\chi_{\midd}(D)|D|^{\sigma}u(t,\cdot)\|_{\dot{H}^s(\mb{R}^n)}+\|\chi_{\midd}(D)u_t(t,\cdot)\|_{\dot{H}^s(\mb{R}^n)}\lesssim \mathrm{e}^{-ct}\|u_1\|_{L^2(\mb{R}^n)}.
	\end{align*}
	Summarizing the derived estimates completes the proof.
\end{proof}

Besides the achieved estimates from Theorem \ref{Thm_Est_Solution_Derivative}, we may follow the proof of previous theorems to arrive at the following further estimates.

\begin{coro}\label{Cor_Est_Solution_Derivative}
	Let us assume $u_1\in H^{s+\ell}(\mb{R}^n)\cap L^m(\mb{R}^n)$ with $s\geqslant0$, $\ell\geqslant0$ and $m\in[1,2]$. Then, the solution $u=u(t,x)$ to \eqref{Linear_Generalized_Plate_Memory} fulfills the estimates as follows:
	\begin{align*}
	\|u(t,\cdot)\|_{\dot{H}^{s}(\mb{R}^n)}&\lesssim 
	\begin{cases}
	(1+t)^{-\frac{n}{2\sigma}(\frac{1}{m}-\frac{1}{2})-\frac{s}{2\sigma}} \|u_1\|_{ L^m(\mb{R}^n)}+(1+t)^{-\frac{1}{2}-\frac{\ell}{2\sigma}+\frac{s}{2\sigma}}\|u_1\|_{H^{\ell}(\mb{R}^n)}, \\
	(1+t)^{-\frac{n}{2\sigma}(\frac{1}{m}-\frac{1}{2})-\frac{s}{2\sigma}} \|u_1\|_{ L^m(\mb{R}^n)}+(1+t)^{-\frac{1}{2}-\frac{\ell}{2\sigma}}\|u_1\|_{H^{s+\ell}(\mb{R}^n)},
	\end{cases}
	\end{align*}
	for any $t\geqslant0$. 
\end{coro}

\begin{remark}
	Comparing with the previous study on the Cauchy problem for plate equation with an exponential decay memory term, i.e. Theorem 2.7 in \cite{Liu-Ueda-2020}, our estimates stated in Theorem \ref{Thm_Est_Solution_Derivative} with $\sigma=2$ for the energies are sharper because we derived the estimates for $|\xi|^{\sigma}\hat{u}(t,\xi)$ and $\hat{u}_t(t,\xi)$ independently in the Fourier space. Furthermore, we also derive estimates for the solution itself in Theorem \ref{Thm_Est_Solution_Derivative}, which did not be shown in \cite{Liu-Ueda-2020}.
\end{remark}
\begin{remark}
	According to Theorems \ref{Thm_Energy_U} and \ref{Thm_Est_Solution_Derivative}, we may observe the decay properties of regularity-loss if $\ell>0$, namely, it requires $s+\ell$ regularity for initial data to estimate the energies in the $\dot{H}^s$ norm. In the special case when $\ell=0$, the effect of regularity-loss will disappear, nevertheless, the decay properties also will lost simultaneously.
\end{remark}
\begin{remark}\label{Compare-Duong-Reissig}
	Let us now compare the energy estimates with initial data taken from $L^1\cap L^2$ space between Theorem \ref{Thm_Est_Solution_Derivative} and those from the previous study on the Cauchy problem for $\sigma$-evolution models with friction, namely,
	\begin{align}\label{Sigma-Evo_Friction}
	\begin{cases}
	u_{tt}+(-\Delta)^{\sigma}u+u_t=0,&x\in\mb{R}^n,\ t>0,\\
	u(0,x)=0,\ u_t(0,x)=u_1(x),&x\in\mb{R}^n.
	\end{cases}
	\end{align}
	The authors in \cite{Duong-Reissig-2017} proved the estimates for solutions to \eqref{Sigma-Evo_Friction} such that 
	\begin{align*}
	\|u(t,\cdot)\|_{L^2(\mb{R}^n)}&\lesssim (1+t)^{-\frac{n}{4\sigma}}\|u_1\|_{ L^1(\mb{R}^n)}+\mathrm{e}^{-ct}\|u_1\|_{H^{-\sigma}(\mb{R}^n)},\\
	\|\,|D|^{\sigma}u(t,\cdot)\|_{L^2(\mb{R}^n)}&\lesssim (1+t)^{-\frac{n}{4\sigma}-\frac{1}{2}}\|u_1\|_{L^1(\mb{R}^n)}+\mathrm{e}^{-ct}\|u_1\|_{L^2(\mb{R}^n)},\\
	\|u_t(t,\cdot)\|_{L^2(\mb{R}^n)}&\lesssim (1+t)^{-\frac{n}{4\sigma}-1}\|u_1\|_{ L^1(\mb{R}^n)}+\mathrm{e}^{-ct}\|u_1\|_{L^2(\mb{R}^n)},
	\end{align*}
	with a positive constant $c$, which are stronger than ours in Theorem \ref{Thm_Est_Solution_Derivative}, in particular, the time-dependent coefficient of the $H^{-\sigma}$ or $L^2$ norms of $u_1$. Nonetheless, the estimates for solutions localized to small frequency zone are coincide with those stated in Theorem \ref{Thm_Est_Solution_Derivative}, which are represented by the time-dependent coefficient of the $L^1$ norm of $u_1$. As a consequence, we observe the damping effect generated by $u-g\ast u$ containing the exponential decay memory term  is weaker than that generated by the frictional damping term $u_t$ due to the regularity-loss effect.
\end{remark}

\subsection{Generalized diffusion phenomena}
It is well-known that the diffusion phenomenon allows us to bridge decay behaviors of solutions between the dissipative evolution equation and its corresponding parabolic-like evolution equations. For example, the classical diffusion phenomenon bridges decay behaviors between damped wave equation and heat equation. Let us turn to the general evolution equations. Our aim is to investigate general diffusion phenomena, which gives approximate equations to describe the general evolution equations.

Let us begin with introducing some notations to state generalized diffusion phenomena, particularly, the approximate equations and their solutions. 

For one thing, we denote $|\xi|$-dependent functions
\begin{align*}
\lambda_j^0(|\xi|):=\begin{cases}
-|\xi|^{2\sigma}&\text{if}\ \ j=1,\\
-\left(\tfrac{1}{2}+i\tfrac{\sqrt{3}}{2}\right)+\left(\tfrac{1}{2}-i\tfrac{\sqrt{3}}{6}\right)|\xi|^{2\sigma}&\text{if}\ \  j=2,\\
-\left(\tfrac{1}{2}-i\tfrac{\sqrt{3}}{2}\right)+\left(\tfrac{1}{2}+i\tfrac{\sqrt{3}}{6}\right)|\xi|^{2\sigma}&\text{if}\ \ j=3,
\end{cases}
\end{align*}
and
\begin{align*}
\lambda_j^{\infty}(|\xi|):=\begin{cases}
i|\xi|^{\sigma}+\tfrac{i}{2}|\xi|^{-\sigma}-\tfrac{1}{2}|\xi|^{-2\sigma}&\text{if}\ \ j=1,\\
-i|\xi|^{\sigma}-\tfrac{i}{2}|\xi|^{-\sigma}-\tfrac{1}{2}|\xi|^{-2\sigma}&\text{if}\ \ j=2,\\
-1+|\xi|^{-2\sigma}&\text{if}\ \ j=3.
\end{cases}
\end{align*}
Obviously, $\lambda_j^0(|\xi|)$ and $\lambda_j^{\infty}(|\xi|)$ with $j=1,2,3$ are the principal parts of the corresponding eigenvalues $\lambda_j(|\xi|)$ for small frequencies and for large frequencies, respectively. That is to say for $j=1,2,3$
\begin{align*}
\lambda_j(|\xi|)-\lambda_j^0(|\xi|)&=\ml{O}(|\xi|^{3\sigma})\ \ \ \, \mbox{as}\ \ |\xi|\to 0,\\
\lambda_j(|\xi|)-\lambda_j^{\infty}(|\xi|)&=\ml{O}(|\xi|^{-3\sigma})\ \ \mbox{as} \ \ |\xi|\to\infty.
\end{align*}
These remaining parts contribute to the additional decay rate or the improvement for regularity on the estimates.

For another, we introduce some reference equations motivated by the eigenvalues $\lambda_j^0(|\xi|)$ and $\lambda_j^{\infty}(|\xi|)$ for $j=1,2,3$. The reference equations with respect to $U^0=U^0(t,x)$ can be written by
\begin{align}\label{Ref_Eq_Small}
\begin{cases}
U^{0}_t+\Lambda_0U^{0}+\Lambda_1(-\Delta)^{\sigma}U^{0}=0,&x\in\mb{R}^n,\ t>0,\\
U^0(0,x):=\ml{F}^{-1}\left((I_{3}+N_2(|\xi|))^{-1}\right)(x) N_1^{-1}U_0(x)&x\in\mb{R}^n,
\end{cases}
\end{align}
where the diagonal coefficient matrices are defined by
\begin{align*}
\Lambda_0:=\diag\left(0,\tfrac{1}{2}+i\tfrac{\sqrt{3}}{2},\tfrac{1}{2}-i\tfrac{\sqrt{3}}{2}\right),\quad \Lambda_1:=\diag\left(1,-\tfrac{1}{2}+i\tfrac{\sqrt{3}}{6},-\tfrac{1}{2}-i\tfrac{\sqrt{3}}{6}\right),
\end{align*}
and the matrices $N_1$, $N_2(|\xi|)$ were defined in the statement of Proposition \ref{Prop_Repre_Small}. So, the solution to \eqref{Ref_Eq_Small} in the Fourier space can be uniquely represented by $\lambda_j^0(|\xi|)$ with $N_1$ and $N_2(|\xi|)$.\\ The reference equations with respect to $U^{\infty}=U^{\infty}(t,x)$  can be written by
\begin{align}\label{Ref_Eq_Large}
\begin{cases}
U^{\infty}_t+\Lambda_2(-\Delta)^{\sigma/2}U^{\infty}+\Lambda_3U^{\infty}+\Lambda_4(-\Delta)^{-\sigma/2}U^{\infty}+\Lambda_5(-\Delta)^{-\sigma}U^{\infty}=0,&x\in\mb{R}^n, \ t>0,\\
U^{\infty}(0,x):=\ml{F}^{-1}\left((I_{3}+N_5(|\xi|))^{-1}(I_{3}+N_4(|\xi|))^{-1}\right)(x) U_0(x),&x\in\mb{R}^n,
\end{cases}
\end{align}
where the diagonal coefficient matrices are defined by
\begin{align*}
\Lambda_2:=\diag(-i,i,0),\quad\Lambda_3:=\diag(0,0,1),\quad\Lambda_4:=\diag\left(-\tfrac{i}{2},\tfrac{i}{2},0\right),\quad\Lambda_5:=\diag\left(\tfrac{1}{2},\tfrac{1}{2},-1\right),
\end{align*}
and the matrices $N_4(|\xi|)$, $N_5(|\xi|)$ were defined in the statement of  Proposition \ref{Prop_Repre_Large}. In other words, the solution to \eqref{Ref_Eq_Large} in the Fourier space is uniquely shown by $\lambda_j^{\infty}(|\xi|)$ with $N_4(|\xi|)$ and $N_5(|\xi|)$.

Before showing our main result, let us denote
\begin{align*}
S_0(t,x)&:=\chi_{\intt}(D)N_1\ml{F}^{-1}\left(I_3+N_2(|\xi|)\right)(x) U^0(t,x),\\
S_{\infty}(t,x)&:=\chi_{\extt}(D)\ml{F}^{-1}\left((I_3+N_4(|\xi|)(I_3+N_5(|\xi|))\right)(x) U^{\infty}(t,x),
\end{align*}
where $U^0(t,x)$ is the solution to \eqref{Ref_Eq_Small}, and
$U^{\infty}(t,x)$ is the solution to \eqref{Ref_Eq_Large}.
The functions $S_0(t,x)$ and $S_{\infty}(t,x)$ are useful for us to describe generalized diffusion phenomena later.

From Theorem \ref{Thm_Energy_U}, concerning the estimation of $U(t,\cdot)$ in the $\dot{H}^s$ spaces with $s\geqslant 0$, the decay rate and the regularity for initial data are determined by small frequencies and large frequencies, respectively. In other words, we may explain generalized diffusion phenomena by the behavior of Fourier multipliers localized to small and large frequencies. One may follow the same approach to Theorem 4.1 in \cite{Liu-Chen-2020} to get the next theorem. Thus, we will omit the proof.
\begin{theorem}\label{Thm_Diff_Phom}
	Let us assume $U_0\in (H^{s+\ell}(\mb{R}^n)\cap L^m(\mb{R}^n))^3$ with $s\geqslant0$, $\ell\geqslant0$ and $m\in[1,2]$. Then, the solution $U=U(t,x)$ to \eqref{Eq_First_Order_x} fulfills the refinement estimates as follows:
	\begin{align*}
	\|(U-S_0)(t,\cdot)\|_{(\dot{H}^{s}(\mb{R}^n))^3} &\lesssim
	(1+t)^{-\frac{n}{2\sigma}(\frac{1}{m}-\frac{1}{2})-\frac{1}{2}}\|U_0\|_{(L^m(\mb{R}^n))^3}+(1+t)^{-\frac{\ell}{2\sigma}}\|U_0\|_{(H^{s+\ell}(\mb{R}^n))^3},\\
	\|(U-S_{\infty})(t,\cdot)\|_{(\dot{H}^{s}(\mb{R}^n))^3} &\lesssim
	(1+t)^{-\frac{n}{2\sigma}(\frac{1}{m}-\frac{1}{2})}\|U_0\|_{(L^m(\mb{R}^n))^3}+(1+t)^{-\frac{\ell}{2\sigma}}\|U_0\|_{(H^{s+\ell-\sigma}(\mb{R}^n))^3},\\
	\|(U-S_0-S_{\infty})(t,\cdot)\|_{(\dot{H}^{s}(\mb{R}^n))^3} &\lesssim
	(1+t)^{-\frac{n}{2\sigma}(\frac{1}{m}-\frac{1}{2})-\frac{1}{2}}\|U_0\|_{(L^m(\mb{R}^n))^3}+(1+t)^{-\frac{\ell}{2\sigma}}\|U_0\|_{(H^{s+\ell-\sigma}(\mb{R}^n))^3},
	\end{align*}
	for $s+\ell-\sigma\geqslant0$ and for any $t\geqslant0$.
\end{theorem}
\begin{remark}
	Let us now compare the decay rate and regularity of initial data stated in Theorem \ref{Thm_Energy_U} with those in Theorem \ref{Thm_Diff_Phom}. We find that the decay rate for initial data taken from $L^m$ spaces in Theorem \ref{Thm_Diff_Phom} can be improved by $(1+t)^{-1/2}$ by subtracting $S_0(t,x)$. Additionally, by keeping the decay rate $(1+t)^{-\ell/(2\sigma)}$ the regularity of initial data in Theorem \ref{Thm_Diff_Phom} can be weakened by $\sigma$ order by subtracting $S_{\infty}(t,x)$.
\end{remark}
\begin{remark}
	According to Theorem \ref{Thm_Diff_Phom} and the approximate equations \eqref{Ref_Eq_Small} as well as \eqref{Ref_Eq_Large}, we may interpret this effect as doubly diffusion phenomena, whose reference equations consist of two different approximate equations. The doubly diffusion phenomena were firstly studied by \cite{D'Ab-Ebert-2014}.
\end{remark}
\begin{remark}
	Actually, one may assume that $u(0,x)=u_0(x)$ is not identically zero in the linear $\sigma$-evolution models \eqref{Linear_Generalized_Plate_Memory}. By processing the same calculation, one also can obtain the same results for energy estimates and generalized diffusion phenomena as in Theorem \ref{Thm_Energy_U} and Theorem \ref{Thm_Diff_Phom}, respectively.
\end{remark}
\begin{remark}
	Even taking $\sigma=2$ in Theorem \ref{Thm_Diff_Phom}, it is not reasonable to compare this result with the results in \cite{Liu-Ueda-2020}. In the recent paper \cite{Liu-Ueda-2020}, the authors derived large time approximation and found the approximation equation by concerning another ansatz
	\begin{align*}
	\widetilde{U}(t,x):=\left(\Delta u(t,x),u_t(t,x),(g\ast u)(t,x)-u(t,x)\right)^{\mathrm{T}}.
	\end{align*}
	What we do is to analyze the ansatz $U(t,x)$ defined in \eqref{Formu_U}. Hence, we are investigating a different change of variables.
\end{remark}

\section{Global (in time) existence of small data solutions}\label{Section_GESDS}
\setcounter{equation}{0}
Let us first introduce some notations which will be used in the proof of global (in time) existence of small data solutions. We denote $E=E(t,x)$ as the fundamental solution to the linear Cauchy problem \eqref{Linear_Generalized_Plate_Memory} with initial data $u(0,x)=0$ and $u_t(0,x)=\delta_0$, where $\delta_0$ is the Dirac distribution at $x=0$ with respect to the spatial variables.

Before introducing our aim, let us show the existence of the fundamental solution $E(t,x)$. It seems not trivial due to  the time-convolution term in the linear part. In the previous research \cite{Liu-Ueda-2020}, the authors considered the special case $\sigma=2$ in \eqref{Linear_Generalized_Plate_Memory} and they combined the Fourier transform with the Laplace transform to get the fundamental solution. To prove the existence of the inverse Laplace transform of the kernel functions, the authors studied the zero points. Here, one recognizes that we will use a simple proof to indicate the existence of $E(t,x)$ with the exponential decay kernel. The main idea is to get the solution to the second-order (with respect to $t$) evolution equation with memory term  from the third-order (with respect to $t$) evolution equation without memory term, which allows us to avoid the memory term in the treatment of the Laplace transform.  It is clear that the fundamental solution to 
\begin{align}\label{Fund_E_3}
\begin{cases}
\widetilde{E}_{ttt}+\widetilde{E}_{tt}+(-\Delta)^{\sigma}\widetilde{E}_t+\widetilde{E}_t+(-\Delta)^{\sigma}\widetilde{E}=0,&x\in\mb{R}^n,\ t>0,\\
\widetilde{E}(0,x)=0,\ \widetilde{E}_t(0,x)=\delta_0,\ \widetilde{E}_{tt}(0,x)=0,&x\in\mb{R}^n,
\end{cases}
\end{align}
exists by using the partial Fourier transform and its inverse transform with respect to spatial variables. According to the relation \eqref{Eq_derivative_g*u}, we may write the equation in \eqref{Fund_E_3} by
\begin{align*}
\left(\partial_t+\ml{I}\right)\left(\widetilde{E}_{tt}+(-\Delta)^{\sigma}\widetilde{E}+\widetilde{E}-g\ast\widetilde{E}\right)=0,
\end{align*}
where $\ml{I}$ denotes the unit operator mapping the function to itself. Then, by multiplying the above equation by $\mathrm{e}^{t}$ and integrating the resultant over $[0,t]$, we derive
\begin{align}\label{Fund_E_1}
\begin{cases}
\widetilde{E}_{tt}+(-\Delta)^{\sigma}\widetilde{E}+\widetilde{E}-g\ast \widetilde{E}=0,&x\in\mb{R}^n,\ t>0,\\
\widetilde{E}(0,x)=0,\ \widetilde{E}_t(0,x)=\delta_0,&x\in\mb{R}^n,
\end{cases}
\end{align}
whose solution is exactly the fundamental solution to \eqref{Linear_Generalized_Plate_Memory}. It is clear that the existence of $\widetilde{E}(t,x)$ implies immediately the existence of $E(t,x)$. Again, one may also strictly follow the proof of Lemma 2.1 in \cite{Liu-Ueda-2020} to prove the existence of the fundamental solution to \eqref{Linear_Generalized_Plate_Memory} without any new difficulty. More precisely, we just need to change $|\xi|^{4}$ in the proof of Lemma 2.1 in \cite{Liu-Ueda-2020} by $|\xi|^{2\sigma}$ with $\sigma\geqslant1$.

We now may represent the solution to the corresponding linear equation \eqref{Linear_Generalized_Plate_Memory} by
\begin{align*}
u^{\lin}(t,x):=E(t,x)\ast_{(x)}u_1(x).
\end{align*}
For $T>0$, we define the operator $N$ such that
\begin{align*}
N:\ u\in X(T)\to Nu(t,x):=u^{\lin}(t,x)+u^{\non}(t,x),
\end{align*}
where $X(T)$ is an evolution space to be determined later. Moreover, having in mind Duhamel's principle the integral operator $u^{\non}(t,x)$ can be shown by
\begin{align*}
u^{\non}(t,x):=\int_0^tE(t-\tau,x)\ast_{(x)}|u(\tau,x)|^p\mathrm{d}\tau.
\end{align*}
Then, we will prove global (in time) solution to the semilinear problem \eqref{Semi_Generalized_Plate_Memory} as a fixed point of the operator $N$. Consequently, we will prove the following two crucial inequalities in the next parts:
\begin{align}
\|Nu\|_{X(T)}&\lesssim \|u_1\|_{\ml{D}(\mb{R}^n)}+\|u\|_{X(T)}^p,\label{Eq_BDD}\\
\|Nu-Nv\|_{X(T)}&\lesssim \|u-v\|_{X(T)}\left(\|u\|_{X(T)}^{p-1}+\|v\|_{X(T)}^{p-1}\right),\label{Eq_LIP}
\end{align}
with the aim of demonstrating the global (in time) existence and uniqueness of small data solutions. Here, the data space $\ml{D}(\mb{R}^n)$ will be fixed in the statement of the theorem. Providing that $\|u_1\|_{\ml{D}(\mb{R}^n)}=\epsilon$ is sufficiently small, then the estimates \eqref{Eq_BDD} and \eqref{Eq_LIP} result the existence of a unique local and global (in time) solution in $X(T)$ by using Banach's fixed point theorem.

In order to prove global existence result, some tools from Harmonic Analysis shown in Appendix \ref{Harmonic_Anal} combined with the following lemma come into play.
\begin{lemma} \label{LemmaIntegral}
	The following inequalities:
	\begin{align*}
	\int_0^t (1+t-\tau)^{-\alpha}(1+\tau)^{-\beta}\mathrm{d}\tau \lesssim
	\begin{cases}
	(1+t)^{-\min\{\alpha, \beta\}} &\mbox{if}\ \ \max\{\alpha, \beta\}>1, \\
	(1+t)^{-\min\{\alpha, \beta\}}\log(\mathrm{e}+t)  &\mbox{if}\ \  \max\{\alpha, \beta\}=1, \\
	(1+t)^{1-\alpha-\beta} &\mbox{if}\ \ \max\{\alpha, \beta\}<1,
	\end{cases} 
	\end{align*}
	hold for all $\alpha, \beta \in \R$.
\end{lemma}

Let us state our main result for global (in time) existence of small data solutions.

\begin{theorem} \label{theorem3.1}
	Let $\sigma\geqslant 1$, $m \in [1,2)$ and $s \in (0,\sigma)$. Let $\ell$ being subject to $\max\{\ell^{*},0\}\leqslant \ell<s$, where $\ell^{*}:=n(\frac{1}{m}-\frac{1}{2})+2s-\sigma$. We suppose that $p> 1+ \lceil \ell \rceil$ and $n<\tfrac{2m\sigma}{2-m}$ fulfilling $p\geqslant \frac{2}{m}$ for all $n\leqslant \frac{4s}{2-m}$, and additionally, $p\leqslant\frac{n-2\ell}{n-2s}$ if $n>2s$. Providing that
	\begin{align}\label{Critical_exponent}
	p>p_{\mathrm{crit}}(n,m,\sigma)= 1+\frac{2m\sigma}{n},
	\end{align}
	then there exists a sufficiently small constant $\epsilon>0$ such that for any data $ u_1 \in \mathcal{D}(\mb{R}^n):= L^m(\mb{R}^n) \cap H^{s+\ell}(\mb{R}^n)$ satisfying the assumption $\|u_1\|_{\mathcal{D}(\mb{R}^n)} \leqslant \epsilon$, there is a uniquely determined global (in time) Sobolev solution
	\begin{align*}
	u \in \ml{C}\left([0,\ity),H^s(\mb{R}^n)\right)
	\end{align*}
	to \eqref{Semi_Generalized_Plate_Memory}. Moreover, the following estimates hold:
	\begin{align*}
	\|u(t,\cdot)\|_{L^2(\mb{R}^n)}&\lesssim (1+t)^{-\frac{n}{2\sigma}(\frac{1}{m}-\frac{1}{2})} \|u_1\|_{ \mathcal{D}(\mb{R}^n)},\\
	\|\,|D|^s u(t,\cdot)\|_{L^2(\mb{R}^n)}&\lesssim (1+t)^{-\frac{n}{2\sigma}(\frac{1}{m}-\frac{1}{2})-\frac{s}{2\sigma}}\|u_1\|_{\mathcal{D}(\mb{R}^n)}.
	\end{align*}
\end{theorem}

\begin{exam}
	Let us consider the case when $m=1$ in Theorem \ref{theorem3.1}. Moreover, we take $2s=\sigma+1-n/2$ so that $\ell=\ell^*=1$. We assume initial data $u_1$ belonging to $L^1(\mb{R}^n)\cap H^{(\sigma+3)/2-n/4}(\mb{R}^n)$ and being sufficiently small. Then, there exists a unique determined Sobolev solution
	\begin{align*}
	u\in\ml{C}\left([0,\infty),H^{(\sigma+1)/2-n/4}(\mb{R}^n)\right)
	\end{align*}
	to \eqref{Semi_Generalized_Plate_Memory} with $\sigma>n/2+1$ and $\sigma\geqslant n-1$, providing that the exponent $p$ of the power nonlinearity satisfies the next conditions:
	\begin{itemize}
		\item for $n=1,2$, we consider $p>p_{\mathrm{crit}}(n,1,\sigma)$;
		\item for $n\geqslant 3$, we consider $p>p_{\mathrm{crit}}(n,1,\sigma)$, and additionally, $p\leqslant\frac{2n-4}{3n-2(\sigma+1)}$ if $\frac{2(\sigma+1)}{3}<n<n_0(\sigma)$, where $n_0(\sigma)$ denotes the positive root of the quadratic equation $n^2+2(2\sigma+1)n-4(\sigma^2+\sigma)=0$. It is obvious to verify $\frac{2(\sigma+1)}{3}<n_0(\sigma)$ holds for any $\sigma\geqslant 2$ and $n\geqslant 3$.
	\end{itemize}
	We want to underline that the condition $n<n_0(\sigma)$ for $n\geqslant3$ is to guarantee the nonempty interval for $p$ such that $p\in(p_{\mathrm{crit}}(n,1,\sigma),\frac{2n-4}{3n-2(\sigma+1)}]$. Moreover, due to $\sigma>n/2+1$, the condition $n>2s$ implies the condition $n>2$, which is the reason for dividing our consideration into the cases $n=1,2$ and $n\geqslant 3$.
\end{exam}
\begin{remark}\label{nonempty s}
	The restriction $\ell^*\leqslant \ell<s$ leads to $s<\sigma- n(\frac{1}{m}-\frac{1}{2})$. Concerning the existence of parameter $s \in (0,\sigma)$, we would like to say
	\begin{align*}
	0<\sigma-n\left(\frac{1}{m}-\frac{1}{2}\right)\ \ \mbox{iff}\ \ n<\frac{2m\sigma}{2-m}.
	\end{align*}
	In other words, the restriction on dimensions in Theorem \ref{theorem3.1} at least should be $n<\frac{2m\sigma}{2-m}$.
\end{remark}
\begin{remark}
	Actually, by fixing some parameters $\sigma$ and $n$ in Theorem \ref{theorem3.1}, there are a lot of examples showing that the lower bound of the exponent $p$ is given by $p_{\mathrm{crit}}(n,m,\sigma)$ if we choose some suitable parameters on the regularities  $s,m,\ell$ of initial data. For example, if we take $\sigma=2$ and $n=1$, then we arrive at $p>p_{\mathrm{crit}}(1,1,2)=5$ with the choice of $s=1$, $m=1$ and $\ell=1/2$.
\end{remark}
\begin{remark}
	In the case when $n>2s$ in Theorem \ref{theorem3.1}, we may weaken the upper bound for the exponent $p$ and the dimension $n$ to obtain their larger admissible range, e.g. we may consider $s-n/2\to 0^-$ and $m-2\to 0^-$, respectively. We should pay attention on other restrictions of regularity $s$, e.g. $s\in(0,\sigma)$.
\end{remark}
\begin{remark}
	Inherited by the linear Cauchy problem \eqref{Linear_Generalized_Plate_Memory}, we also observe the effect of regularity-loss in Theorem \ref{theorem3.1}. Precisely, to derive global (in time) existence of Sobolev solutions with $H^s$ regularity, we need to assume initial data having $H^{s+\ell}$ regularity.
\end{remark}

\begin{proof}
	We define the solution space for $T>0$ by
	$$ X(T):= \ml{C}\left([0,T],H^s(\mb{R}^n)\right), $$
	equipped with the corresponding norm
	$$ \|u\|_{X(T)}:= \sup_{0\leqslant t \leqslant T} \left((1+t)^{\frac{n}{2\sigma}(\frac{1}{m}-\frac{1}{2})}\|u(t,\cdot)\|_{L^2(\mb{R}^n)}+ (1+t)^{\frac{n}{2\sigma}(\frac{1}{m}-\frac{1}{2})+\frac{s}{2\sigma}}\|\,|D|^s u(t,\cdot)\|_{L^2(\mb{R}^n)}\right). $$
	
	Before indicating the both inequalities \eqref{Eq_BDD} and \eqref{Eq_LIP}, we need the following auxiliary estimates:
	\begin{align}
	\|\,|u(\tau,\cdot)|^p\|_{L^m(\mb{R}^n)}= \|u(\tau,\cdot)\|^p_{L^{mp}(\mb{R}^n)} &\lesssim (1+\tau)^{-\frac{n}{2m\sigma}(p-1)}\|u\|^p_{X(\tau)}, \label{t3.1.1}\\ 
	\|\,|u(\tau,\cdot)|^p\|_{L^2(\mb{R}^n)}= \|u(\tau,\cdot)\|^p_{L^{2p}(\mb{R}^n)} &\lesssim (1+\tau)^{-\frac{n}{2\sigma}(\frac{p}{m}-\frac{1}{2})}\|u\|^p_{X(\tau)}, \label{t3.1.2}\\
	\|\,|u(\tau,\cdot)|^p\|_{\dot{H}^\ell(\mb{R}^n)} &\lesssim (1+\tau)^{-\frac{n}{2\sigma}(\frac{p}{m}-\frac{1}{2})-\frac{\ell}{2\sigma}}\|u\|^p_{X(\tau)}. \label{t3.1.3}
	\end{align}
	Indeed, the first two estimates in the above can be directly obtained from the application of the fractional Gagliardo-Nirenberg inequality from Proposition \ref{fractionalgagliardonirenbergineq} and the definition of the evolution space, provided that
	$$ p \in \left[\frac{2}{m}, \ity \right) \ \ \text{if} \  \  n \leqslant 2s, \ \ \text{or}\ \  p \in \left[\frac{2}{m}, \frac{n}{n-2s}\right] \ \  \text{if}\ \  n >2s. $$
	To prove \eqref{t3.1.3}, employing the fractional chain rule from Proposition \ref{fractionalchainrule} with $p >\lceil \ell \rceil$ we can proceed as follows:
	$$\|\,|u(\tau,\cdot)|^p\|_{\dot{H}^\ell(\mb{R}^n)}\lesssim \|u(\tau,\cdot)\|^{p-1}_{L^{r_1}(\mb{R}^n)}\,\|u(\tau,\cdot)\|_{\dot{H}^\ell_{r_2}(\mb{R}^n)},$$
	where $\tfrac{1}{2}=\tfrac{p-1}{r_1}+\tfrac{1}{r_2}$ with $1<r_1,r_2<\infty$.\\
	Then, applying the fractional Gagliardo-Nirenberg inequality from Proposition \ref{fractionalgagliardonirenbergineq} gives
	\begin{align*}
	\|u(\tau,\cdot)\|_{L^{r_1}(\mb{R}^n)} \lesssim \|u(\tau,\cdot)\|^{1-\theta_{r_1}}_{L^2(\mb{R}^n)}\,\|\,|D|^s u(\tau,\cdot)\|^{\theta_{r_1}}_{L^2(\mb{R}^n)}&\lesssim (1+\tau)^{-\frac{n}{2\sigma}(\frac{1}{m}-\frac{1}{2})-\frac{s}{2\sigma}\theta_{r_1}}\|u\|_{X(\tau)} \\ 
	&\lesssim (1+\tau)^{-\frac{n}{2\sigma}(\frac{1}{m}-\frac{1}{r_1})}\|u\|_{X(\tau)},
	\end{align*}
	where $\theta_{r_1}=\frac{n}{s}(\frac{1}{2}-\frac{1}{r_1})$, and similarly
	\begin{align*}
	\|u(\tau,\cdot)\|_{\dot{H}^\ell_{r_2}(\mb{R}^n)}\lesssim \|u(\tau,\cdot)\|^{1-\theta_{r_2}}_{L^2(\mb{R}^n)}\,\|\,|D|^s u(\tau,\cdot)\|^{\theta_{r_2}}_{L^2(\mb{R}^n)}&\lesssim (1+\tau)^{-\frac{n}{2\sigma}(\frac{1}{m}-\frac{1}{2})-\frac{s}{2\sigma}\theta_{r_2}}\|u\|_{X(\tau)} \\ 
	&\lesssim (1+\tau)^{-\frac{n}{2\sigma}(\frac{1}{m}-\frac{1}{r_2})-\frac{\ell}{2\sigma}}\|u\|_{X(\tau)},
	\end{align*}
	where $\theta_{r_2}=\frac{n}{s}(\frac{1}{2}-\frac{1}{r_2}+\frac{\ell}{n})$ and we considered $\ell\leqslant s$. For this reason, we may conclude by summing up the previous derived inequalities
	$$\|\,|u(\tau,\cdot)|^p\|_{\dot{H}^\ell(\mb{R}^n)}\lesssim (1+\tau)^{-\frac{n}{2\sigma}(\frac{p}{m}-\frac{1}{2})-\frac{\ell}{2\sigma}}\|u\|_{X(\tau)}. $$
	Here, we have to guarantee that $\theta_{r_1}\in [0,1]$ and $\theta_{r_2}\in [\ell/s,1]$, which imply the restrictions
	$$1<p \leqslant \frac{n-2\ell}{n-2s}\ \  \text{if}\ \  n>2s, \quad \text{or}\quad p>1 \ \ \text{if}\ \ n \leqslant 2s. $$
	
	\textit{First let us prove the inequality \eqref{Eq_BDD}.} From the estimates for solutions to the linear Cauchy problem \eqref{Linear_Generalized_Plate_Memory}, which are shown in Theorem \ref{Thm_Est_Solution_Derivative} and Corollary \ref{Cor_Est_Solution_Derivative}, one may derive
	\begin{align*}
	\|u^{\lin}(t,\cdot)\|_{L^2(\mb{R}^n)}&\lesssim (1+t)^{-\frac{n}{2\sigma}(\frac{1}{m}-\frac{1}{2})}\|u_1\|_{\ml{D}(\mb{R}^n)},\\
	\|\,|D|^su^{\lin}(t,\cdot)\|_{L^2(\mb{R}^n)}&\lesssim (1+t)^{-\frac{n}{2\sigma}(\frac{1}{m}-\frac{1}{2})-\frac{s}{2\sigma}}\|u_1\|_{\ml{D}(\mb{R}^n)},
	\end{align*}
	where we used our assumption $\ell\geqslant\ell^*>\ell^*-s$. For this reason, we immediately claim $u^{\lin}\in X(T)$. From the definition of the data space, it is obvious that we just need to prove the following inequality instead of \eqref{Eq_BDD}:
	\begin{equation}
	\|u^{\non}\|_{X(T)} \lesssim \|u\|^p_{X(T)}. \label{Eq_BDD1}
	\end{equation}
	Our proof is divided into two steps. \medskip
	
	Step 1: $\quad$ We may estimate $u^{\non}(t,\cdot)$ in the $L^2$ norm by applying $(L^2\cap L^m)-L^2$ estimates in $[0,t/2]$ and $L^2-L^2$ estimate in $[t/2,t]$ as follows:
	\begin{align*}
	\|u^{\non}(t,\cdot)\|_{L^2(\mb{R}^n)}&\lesssim \int_0^{t/2} (1+t-\tau)^{-\frac{n}{2\sigma}(\frac{1}{m}-\frac{1}{2})} \|\,|u(\tau,\cdot)|^p\|_{L^m(\mb{R}^n)} \mathrm{d}\tau+ \int_{t/2}^t \|\,|u(\tau,\cdot)|^p\|_{L^2(\mb{R}^n)} \mathrm{d}\tau \\
	&\quad + \int_0^t (1+t-\tau)^{-\frac{1}{2}}\|\,|u(\tau,\cdot)|^p\|_{L^2(\mb{R}^n)} \mathrm{d}\tau \\
	&\lesssim \Big(\int_0^{t/2} (1+t-\tau)^{-\frac{n}{2\sigma}(\frac{1}{m}-\frac{1}{2})}(1+\tau)^{-\frac{n}{2m\sigma}(p-1)} \mathrm{d}\tau+ \int_{t/2}^t (1+\tau)^{-\frac{n}{2\sigma}(\frac{p}{m}-\frac{1}{2})} \mathrm{d}\tau\Big)\|u\|^p_{X(T)} \\
	&\quad +\int_0^t (1+t-\tau)^{-\frac{1}{2}}\|\,|u(\tau,\cdot)|^p\|_{L^2(\mb{R}^n)} \mathrm{d}\tau,
	\end{align*}
	where we have used \eqref{t3.1.1} as well as \eqref{t3.1.2}, and $\|\cdot\|_{X(\tau)}\lesssim\|\cdot\|_{X(T)}$ for any $0\leqslant \tau\leqslant T$. Since $p> p_{\mathrm{crit}}(n,m,\sigma)$, it follows immediately $-\frac{n}{2m\sigma}(p-1)< -1$. For the first two integrals, using the relations $(1+t- \tau) \approx (1+t)$ if $\tau \in [0,t/2]$ and $(1+\tau) \approx (1+t)$ if $\tau \in [t/2,t]$ one derives from the integrability that
	\begin{align*}
	\int_0^{t/2} (1+t-\tau)^{-\frac{n}{2\sigma}(\frac{1}{m}-\frac{1}{2})}(1+\tau)^{-\frac{n}{2m\sigma}(p-1)} \mathrm{d}\tau &\lesssim (1+t)^{-\frac{n}{2\sigma}(\frac{1}{m}-\frac{1}{2})}, \\ 
	\int_{t/2}^t (1+\tau)^{-\frac{n}{2\sigma}(\frac{p}{m}-\frac{1}{2})} \mathrm{d}\tau &\lesssim (1+t)^{-\frac{n}{2\sigma}(\frac{1}{m}-\frac{1}{2})}.
	\end{align*}
	Here, we employed $1-\frac{n}{2\sigma}(\frac{p}{m}-\frac{1}{2})<-\frac{n}{2\sigma}(\frac{1}{m}-\frac{1}{2})$ in the second integral due to our assumption $p>p_{\mathrm{crit}}(n,m,\sigma)$ again.\\
	The applications of Lemma \ref{LemmaIntegral} and \eqref{t3.1.2} lead to the following estimate:
	\begin{align*}
	\int_0^t (1+t-\tau)^{-\frac{1}{2}}\|\,|u(\tau,\cdot)|^p\|_{L^2(\mb{R}^n)} \mathrm{d}\tau&\lesssim \|u\|^p_{X(T)}\int_0^t (1+t-\tau)^{-\frac{1}{2}}(1+\tau)^{-\frac{n}{2\sigma}(\frac{p}{m}-\frac{1}{2})} \mathrm{d}\tau \\ 
	&\lesssim (1+t)^{-\min\{\frac{1}{2},\frac{n}{2\sigma}(\frac{p}{m}-\frac{1}{2})\}}\|u\|^p_{X(T)}\\
	& \lesssim (1+t)^{-\frac{n}{2\sigma}(\frac{1}{m}-\frac{1}{2})}\|u\|^p_{X(T)},
	\end{align*}
	provided that $n\leqslant \frac{2m\sigma}{2-m}$. In the first line of the above chain inequality, we used $p>p_{\mathrm{crit}}(n,m,\sigma)$ so that $\max\{\frac{1}{2},\frac{n}{2\sigma}(\frac{p}{m}-\frac{1}{2})\}=\frac{n}{2\sigma}(\frac{p}{m}-\frac{1}{2})>1$. As a result, combining the above estimates we have proved that
	$$ \|u^{\non}(t,\cdot)\|_{L^2(\mb{R}^n)}\lesssim  (1+t)^{-\frac{n}{2\sigma}(\frac{1}{m}-\frac{1}{2})}\|u\|^p_{X(T)}. $$
	
	Step 2: $\quad$ By using the same ideas as Step 1 and applying some derived estimates from \eqref{t3.1.1} to \eqref{t3.1.3}, we may control the remaining term $|D|^s u^{\non}(t,\cdot)$ in the $L^2$ norm as follows:
	\begin{align*}
	\|\,|D|^s u^{\non}(t,\cdot)\|_{L^2(\mb{R}^n)}&\lesssim \int_0^{t/2} (1+t-\tau)^{-\frac{n}{2\sigma}(\frac{1}{m}-\frac{1}{2})-\frac{s}{2\sigma}} \|\,|u(\tau,\cdot)|^p\|_{L^m(\mb{R}^n)} \mathrm{d}\tau\\
	&\quad + \int_{t/2}^t (1+t-\tau)^{-\frac{s}{2\sigma}}\|\,|u(\tau,\cdot)|^p\|_{L^2(\mb{R}^n)}\mathrm{d}\tau \\
	&\quad+ \int_0^t (1+t-\tau)^{-\frac{1}{2}-\frac{\ell}{2\sigma}+\frac{s}{2\sigma}}\|\,|u(\tau,\cdot)|^p\|_{H^\ell(\mb{R}^n)} \mathrm{d}\tau \\
	&\lesssim \int_0^{t/2} (1+t-\tau)^{-\frac{n}{2\sigma}(\frac{1}{m}-\frac{1}{2})-\frac{s}{2\sigma}}(1+\tau)^{-\frac{n}{2m\sigma}(p-1)} \mathrm{d}\tau\|u\|^p_{X(T)}\\
	&\quad + \int_{t/2}^t (1+t-\tau)^{-\frac{s}{2\sigma}}(1+\tau)^{-\frac{n}{2\sigma}(\frac{p}{m}-\frac{1}{2})} \mathrm{d}\tau \|u\|^p_{X(T)} \\
	&\quad +\int_0^t (1+t-\tau)^{-\frac{\ell}{2\sigma}-\frac{1}{2}+\frac{s}{2\sigma}}\|\,|u(\tau,\cdot)|^p\|_{H^\ell(\mb{R}^n)} \mathrm{d}\tau.
	\end{align*}
	Then, repeating some arguments as we did in Step $1$ we may conclude the following estimates:
	\begin{align*}
	\int_0^{t/2} (1+t-\tau)^{-\frac{n}{2\sigma}(\frac{1}{m}-\frac{1}{2})-\frac{s}{2\sigma}}(1+\tau)^{-\frac{n}{2m\sigma}(p-1)} \mathrm{d}\tau &\lesssim (1+t)^{-\frac{n}{2\sigma}(\frac{1}{m}-\frac{1}{2})-\frac{s}{2\sigma}}, \\
	\int_{t/2}^t (1+t-\tau)^{-\frac{s}{2\sigma}}(1+\tau)^{-\frac{n}{2\sigma}(\frac{p}{m}-\frac{1}{2})} \mathrm{d}\tau &\lesssim (1+t)^{-\frac{n}{2\sigma}(\frac{1}{m}-\frac{1}{2})-\frac{s}{2\sigma}},
	\end{align*}
	by using $s<\sigma$ and $p>p_{\mathrm{crit}}(n,m,\sigma)$.\\
	To deal with the remaining integral, from \eqref{t3.1.2} and \eqref{t3.1.3} we notice that
	$$ \|\,|u(\tau,\cdot)|^p\|_{H^\ell(\mb{R}^n)}\lesssim \|\,|u(\tau,\cdot)|^p\|_{L^2(\mb{R}^n)}+ \|\,|u(\tau,\cdot)|^p\|_{\dot{H}^\ell(\mb{R}^n)}\lesssim (1+\tau)^{-\frac{n}{2\sigma}(\frac{p}{m}-\frac{1}{2})}\|u\|_{X(\tau)}^p. $$
	Consequently, applying Lemma \ref{LemmaIntegral} we may arrive at
	\begin{align*}
	\int_0^t (1+t-\tau)^{-\frac{\ell}{2\sigma}-\frac{1}{2}+\frac{s}{2\sigma}}\|\,|u(\tau,\cdot)|^p\|_{H^\ell(\mb{R}^n)} \mathrm{d}\tau&\lesssim \|u\|^p_{X(T)}\int_0^t (1+t-\tau)^{-\frac{1}{2}-\frac{\ell}{2\sigma}+\frac{s}{2\sigma}}(1+\tau)^{-\frac{n}{2\sigma}(\frac{p}{m}-\frac{1}{2})} \mathrm{d}\tau \\ 
	&\lesssim (1+t)^{-\min\{\frac{1}{2}+\frac{\ell}{2\sigma}-\frac{s}{2\sigma},\,\frac{n}{2\sigma}(\frac{p}{m}-\frac{1}{2})\}}\|u\|^p_{X(T)} \\
	&\lesssim (1+t)^{-\frac{n}{2\sigma}(\frac{1}{m}-\frac{1}{2})-\frac{s}{2\sigma}}\|u\|^p_{X(T)},
	\end{align*}
	where the condition $\ell\geqslant \max\{\ell^{*},0\} $ is fulfilled. Precisely, due to the assumption $\ell<s$ and $p>p_{\mathrm{crit}}(n,m,\sigma)$, we find
	\begin{align*}
	\frac{1}{2}+\frac{\ell}{2\sigma}-\frac{s}{2\sigma}<\frac{n}{2\sigma}\left(\frac{p}{m}-\frac{1}{2}\right)
	\end{align*}
	for any $\sigma\geqslant 1$ and $m\in[1,2)$. Therefore, we obtain
	$$ \|\,|D|^s u^{\non}(t,\cdot)\|_{L^2(\mb{R}^n)}\lesssim (1+t)^{-\frac{n}{2\sigma}(\frac{1}{m}-\frac{1}{2})-\frac{s}{2\sigma}}\|u\|^p_{X(T)}. $$
	From the definition of the norm in $X(T)$ we obtain immediately the inequality \eqref{Eq_BDD1}.
	
	\textit{Next let us prove the inequality \eqref{Eq_LIP}.} We shall follow the strategy used in the proof of the inequality \eqref{Eq_BDD1}. The new difficulty is to require the estimates for the term $|u(\tau,\cdot)|^p-|v(\tau,\cdot)|^p$ in the Lebesgue spaces $L^m$, $L^2$ and the homogeneous Sobolev spaces $\dot{H}^\ell$. Following an analogous treatment as in the proof of the inequality \eqref{Eq_BDD1} we may conclude the inequality \eqref{Eq_LIP}. Indeed, by using H\"{o}lder's inequality we get
	\begin{align*}
	\|\,|u(\tau,\cdot)|^p-|v(\tau,\cdot)|^p\|_{L^m(\mb{R}^n)} &\lesssim \|u(\tau,\cdot)- v(\tau,\cdot)\|_{L^{mp}(\mb{R}^n)} \left(\|u(\tau,\cdot)\|^{p-1}_{L^{mp}(\mb{R}^n)}+\|v(\tau,\cdot)\|^{p-1}_{L^{mp}(\mb{R}^n)}\right), \\ 
	\|\,|u(\tau,\cdot)|^p- |v(\tau,\cdot)|^p\|_{L^2(\mb{R}^n)} &\lesssim \|u(\tau,\cdot)- v(\tau,\cdot)\|_{L^{2p}(\mb{R}^n)} \left(\|u(\tau,\cdot)\|^{p-1}_{L^{2p}(\mb{R}^n)}+\|v(\tau,\cdot)\|^{p-1}_{L^{2p}(\mb{R}^n)}\right).
	\end{align*}
	Analogously to the proof of \eqref{Eq_BDD1}, applying the fractional Gagliardo-Nirenberg inequality from Proposition \ref{fractionalgagliardonirenbergineq} to deal with the norms
	$$ \|u(\tau,\cdot)- v(\tau,\cdot)\|_{L^\eta(\mb{R}^n)}, \quad \|u(\tau,\cdot)\|_{L^\eta(\mb{R}^n)}, \quad \|v(\tau,\cdot)\|_{L^\eta(\mb{R}^n)} $$
	with $\eta=mp$ and $\eta=2p$ we deduce the following estimates:
	\begin{align*}
	\|\,|u(\tau,\cdot)|^p-|v(\tau,\cdot)|^p\|_{L^m(\mb{R}^n)} &\lesssim (1+\tau)^{-\frac{n}{2m\sigma}(p-1)}\|u-v\|_{X(\tau)}\left(\|u\|^{p-1}_{X(\tau)}+\|v\|^{p-1}_{X(\tau)}\right),\\
	\|\,|u(\tau,\cdot)|^p- |v(\tau,\cdot)|^p\|_{L^2(\mb{R}^n)} &\lesssim (1+\tau)^{-\frac{n}{2\sigma}(\frac{p}{m}-\frac{1}{2})} \|u-v\|_{X(\tau)}\left(\|u\|^{p-1}_{X(\tau)}+\|v\|^{p-1}_{X(\tau)}\right).
	\end{align*}
	Let us now turn to estimate the norm $$\|\,|u(\tau,\cdot)|^p-|v(\tau,\cdot)|^p\|_{\dot{H}^\ell(\mb{R}^n)}.$$ By using the integral representation
	$$ |u(\tau,x)|^p-|v(\tau,x)|^p=p\int_0^1 \big(u(\tau,x)-v(\tau,x)\big)G\big(\omega u(\tau,x)+(1-\omega)v(\tau,x)\big)\mathrm{d}\omega, $$
	where $G(u)=u|u|^{p-2}$, we derive
	\begin{align*}
	\|\,|u(\tau,\cdot)|^p-|v(\tau,\cdot)|^p\|_{\dot{H}^\ell(\mb{R}^n)}\lesssim \int_0^1 \big\|\big(u(\tau,\cdot)-v(\tau,\cdot)\big)G\big(\omega u(\tau,\cdot)+(1-\omega)v(\tau,\cdot)\big)\big\|_{\dot{H}^\ell(\mb{R}^n)}\mathrm{d}\omega .
	\end{align*}
	For sake of the fractional Leibniz rule from Proposition \ref{fractionleibnizrule} and Minkowski's inequality, we gain
	\begin{align*}
	&\|\,|u(\tau,\cdot)|^p-|v(\tau,\cdot)|^p\|_{\dot{H}^\ell(\mb{R}^n)}\\
	&\quad \lesssim \|u(\tau,\cdot)-v(\tau,\cdot)\|_{\dot{H}^\ell_{q_1}(\mb{R}^n)} \int_0^1 \big\|G\big(\omega u(\tau,\cdot)+(1-\omega)v(\tau,\cdot)\big)\big\|_{L^{q_2}(\mb{R}^n)}\mathrm{d}\omega\\
	&\qquad+ \big\|u(\tau,\cdot)-v(\tau,\cdot)\|_{L^{q_3}(\mb{R}^n)}  \int_0^1 \big\|G\big(\omega u(\tau,\cdot)+(1-\omega)v(\tau,\cdot)\big)\big\|_{\dot{H}^\ell_{q_4}(\mb{R}^n)}\mathrm{d}\omega\\
	&\quad \lesssim \|u(\tau,\cdot)-v(\tau,\cdot)\|_{\dot{H}^\ell_{q_1}(\mb{R}^n)} \left(\|u(\tau,\cdot)\|^{p-1}_{L^{q_2 (p-1)}(\mb{R}^n)}+ \|v(\tau,\cdot)\|^{p-1}_{L^{q_2 (p-1)}(\mb{R}^n)}\right)\\
	&\qquad+ \|u(\tau,\cdot)-v(\tau,\cdot)\|_{L^{q_3}(\mb{R}^n)} \int_0^1 \big\|G\big(u(\tau,\cdot)+(1-\omega)v(\tau,\cdot)\big)\big\|_{\dot{H}^\ell_{q_4}(\mb{R}^n)}\mathrm{d}\omega,
	\end{align*}
	where $\frac{1}{2}=\frac{1}{q_1}+\frac{1}{q_2}=\frac{1}{q_3}+\frac{1}{q_4}$. Employing the fractional Gargliardo-Nirenberg inequality from Proposition \ref{fractionalgagliardonirenbergineq} follows
	\begin{align*}
	\|u(\tau,\cdot)-v(\tau,\cdot)\|_{\dot{H}^\ell_{q_1}(\mb{R}^n)}&\lesssim \|u(\tau,\cdot)-v(\tau,\cdot)\|^{\theta_1}_{\dot{H}^s(\mb{R}^n)}\,\|u(\tau,\cdot)-v(\tau,\cdot)\|^{1-\theta_1}_{L^2(\mb{R}^n)}, \\
	\|u(\tau,\cdot)\|_{L^{q_2 (p-1)}(\mb{R}^n)}&\lesssim \|u(\tau,\cdot)\|^{\theta_2}_{\dot{H}^s(\mb{R}^n)}\,\|u(\tau,\cdot)\|^{1-\theta_2}_{L^2(\mb{R}^n)}, \\
	\|v(\tau,\cdot)\|_{L^{q_2 (p-1)}(\mb{R}^n)}&\lesssim \|v(\tau,\cdot)\|^{\theta_2}_{\dot{H}^s(\mb{R}^n)}\,\|v(\tau,\cdot)\|^{1-\theta_2}_{L^2(\mb{R}^n)}, \\
	\|u(\tau,\cdot)-v(\tau,\cdot)\|_{L^{q_3}(\mb{R}^n)}&\lesssim \|u(\tau,\cdot)-v(\tau,\cdot)\|^{\theta_3}_{\dot{H}^s(\mb{R}^n)}\,\|u(\tau,\cdot)-v(\tau,\cdot)\|^{1-\theta_3}_{L^2(\mb{R}^n)},
	\end{align*}
	where $\theta_1=\frac{n}{s}(\frac{1}{2}-\frac{1}{q_1}+\frac{\ell}{n}) \in [\ell/s,1],\,\,\theta_2=\frac{n}{s}(\frac{1}{2}-\frac{1}{q_2(p-1)}) \in [0,1]$ and $\theta_3=\frac{n}{s}(\frac{1}{2}-\frac{1}{q_3}) \in [0,1]$.
	Additionally, since $\omega$ is a constant parameter, we may apply again the fractional chain rule from Proposition \ref{fractionalchainrule} with $p> 1+ \lceil \ell \rceil$ and the fractional Gagliardo-Nirenberg inequality from Proposition \ref{fractionalgagliardonirenbergineq} to conclude
	\begin{align*}
	&\|G\big(u(\tau,\cdot)+(1-\omega)v(\tau,\cdot)\big)\|_{\dot{H}^\ell_{q_4}(\mb{R}^n)}\\
	&\quad \lesssim \|\omega u(\tau,\cdot)+(1-\omega) v(\tau,\cdot)\|^{p-2}_{L^{q_5}(\mb{R}^n)}\, \|\omega u(\tau,\cdot)+(1-\omega) v(\tau,\cdot)\|_{\dot{H}^\ell_{q_6}(\mb{R}^n)}\\
	&\quad \lesssim \|\omega u(\tau,\cdot)+(1-\omega) v(\tau,\cdot)\|^{(p-2)\theta_4+\theta_5}_{\dot{H}^s(\mb{R}^n)}\,\|\omega u(\tau,\cdot)+(1-\omega) v(\tau,\cdot)\|^{(p-2)(1-\theta_4)+1-\theta_5}_{L^2(\mb{R}^n)},
	\end{align*}
	where $\frac{1}{q_4}=\frac{p-2}{q_5}+\frac{1}{q_6},\,\,\theta_4=\frac{n}{s}(\frac{1}{2}-\frac{1}{q_5}) \in [0,1]$ and $\theta_5=\frac{n}{s}(\frac{1}{2}-\frac{1}{q_6}+\frac{\ell}{n}) \in [\ell/s,1]$. Hence, we derive
	\begin{align*}
	\int_0^1 \big\|G\big(u(\tau,\cdot)+(1-\omega)v(\tau,\cdot)\big)\big\|_{\dot{H}^\ell_{q_4}(\mb{R}^n)}\mathrm{d}\omega &\lesssim \left(\|u(\tau,\cdot)\|_{\dot{H}^s(\mb{R}^n)}+\|v(\tau,\cdot)\|_{\dot{H}^s(\mb{R}^n)}\right)^{(p-2)\theta_4+\theta_5}\\ &\quad\times\left(\|u(\tau,\cdot)\|_{L^2(\mb{R}^n)}+\|v(\tau,\cdot)\|_{L^2(\mb{R}^n)} \right)^{(p-2)(1-\theta_4)+1-\theta_5}.
	\end{align*}
	Now, combining all previous estimates one gets
	$$\|\,|u(\tau,\cdot)|^p-|v(\tau,\cdot)|^p\|_{\dot{H}^\ell(\mb{R}^n)}\lesssim (1+\tau)^{-\frac{n}{2\sigma}(\frac{p}{m}-\frac{1}{2})-\frac{\ell}{2\sigma}}\|u-v\|_{X(\tau)}\left( \|u\|^{p-1}_{X(\tau)}+ \|v\|^{p-1}_{X(\tau)} \right), $$
	where we notice that $\theta_1+ (p-1)\theta_2= \theta_3+ (p-2)\theta_4+ \theta_5= \frac{n}{s}(\frac{p-1}{2}+\frac{\ell}{n})$. Therefore, all the conditions for $\theta_1,\dots, \theta_5$ and $q_1,\dots,q_6$ lead to
	$$ 1< p \leqslant \frac{n-2\ell}{n-2s}\ \  \text{if}\ \  n>2s, \quad \text{or}\quad p> 1 \ \ \text{if}\ \ n \leqslant 2s. $$
	Finally, we do the straight-forward computation similar to Step 2 to finish the proof of desired inequality \eqref{Eq_LIP}. Thus, our proof is complete.
\end{proof}

\begin{remark}
	Here, we will give an example to verify all the possibility of $r_1$, $r_2$ and $q_1,\dots,q_6$ as required in the proof of Theorem \ref{theorem3.1} by choosing $s=1$, $n=1$, $\sigma=2$ and $m=1$. If we take $\ell=1/2$ and for any $p>5$, then the possible choice of these parameters is as follows:
	\begin{itemize}
		\item $r_1=4(p-1)$ and $r_2=4$,
		\item $q_1=q_2=q_3=q_4=4$, $q_5=8(p-2)$ and $q_6=8$.
	\end{itemize}
\end{remark}

\section{Blow-up of solutions}\label{Section_Blowup}
\setcounter{equation}{0}
Before stating our main result in this section, let us recall the following auxiliary lemmas whose proofs can be found in the recent papers \cite{DaoDuong,Dao-Reissig-2020}. At first, by Lemma 2.1 and Lemma 2.3 in the paper \cite{Dao-Reissig-2020} we may arrive at the following lemma.
\begin{lemma} \label{lemma2.1}
Let $m \in \mb{N}$ and $s \in [0,1)$. Then, the following estimates hold for any $q>n$ and for all $x \in \R^n$:
\begin{align*}
 \left|(-\Delta)^{m+s} \langle x\rangle^{-q}\right| \lesssim
\begin{cases}
\langle x\rangle^{-n-2m} & \text{if}\ \  s=0, \\
\langle x\rangle ^{-n-2s} & \text{if}\ \  s \in (0,1).
\end{cases}
\end{align*}
\end{lemma}
\begin{lemma} \label{lemma2.2} \cite[Lemma~2.2]{DaoDuong}
Let $\gamma \geqslant 1$ be a fractional number. Let $\phi:= \phi(x)= \langle x\rangle^{-q}$ for some $q>0$. For any $R>0$, let $\phi_R$ be a function defined by
$$ \phi_R(x):= \phi(x/R)\ \  \mbox{for all}\ \ x \in \R^n. $$
Then, $(-\Delta)^\gamma (\phi_R)$ satisfies the following scaling properties for all $x \in \R^n$:
\begin{equation*}
(-\Delta)^\gamma (\phi_R)(x)= R^{-2\gamma} \left((-\Delta)^\gamma \phi \right)(x/R).
\end{equation*}
\end{lemma}
\begin{lemma} \label{lemma2.3}\cite[Lemma~2.7]{Dao-Reissig-2020}
Let $s \in \R$. Let $\phi_1=\phi_1(x) \in H^s(\mb{R}^n)$ and $\phi_2=\phi_2(x) \in H^{-s}(\mb{R}^n)$. Then, the following relation holds:
$$ \int_{\R^n}\phi_1(x)\,\phi_2(x)\mathrm{d}x= \int_{\R^n}\hat{\phi}_1(\xi)\,\hat{\phi}_2(\xi)\mathrm{d}\xi. $$
\end{lemma}
Our blow-up result is read as follows.

\begin{theorem} \label{Blow-up.Thoerem}
Let $\sigma\geqslant1$. Let us assume that the initial data $u_1\in L^{m}(\mb{R}^n)$ with $m\in[1,2)$ and fulfills the following conditions:
\begin{equation} \label{u_1.Cond1}
\int_{\mb{R}^n}u_1(x)\mathrm{d}x>0 \quad \text{ if }\quad m=1,
\end{equation}
or
\begin{equation} \label{u_1.Cond2}
u_1(x)\gtrsim|x|^{-\frac{n}{m}}(\log(1+|x|))^{-1} \quad \text{ if }\quad m\in(1,2).
\end{equation}
Then, every local (in time) Sobolev solution $u\in\ml{C}([0,\infty),L^2(\mb{R}^n))$ to \eqref{Semi_Generalized_Plate_Memory} blows up if the exponent $p$ satisfies
\begin{align}\label{Critical exp}
1<p<p_{\mathrm{crit}}(n,m,\sigma)=1+\frac{2m\sigma}{n}
\end{align}
for all $n\geqslant1$.
\end{theorem}
\begin{remark}
Considering the critical case $p=p_{\mathrm{crit}}(n,1,\sigma)$ if $u_1\in L^1(\mb{R}^n)$, one may prove blow-up of weak solutions to the Cauchy problem \eqref{Semi_Generalized_Plate_Memory} with $\sigma\in\mb{N}$ by combining the idea in the next proof and the approach of Theorem 1 in \cite{D'Ab-Ebert-2017}. However, for any fractional number $\sigma\geqslant1$ it is still open to prove a blow-up result in the critical case $p=p_{\mathrm{crit}}(n,1,\sigma)$ and to claim that whether there exists a global (in time) solution or not for $p=p_{\mathrm{crit}}(n,m,\sigma)$ with $m\in (1,2)$, as well.
\end{remark}
\begin{remark}
According to \eqref{Critical_exponent} in Theorem \ref{theorem3.1} and \eqref{Critical exp} in Theorem \ref{Blow-up.Thoerem}, we may claim that the critical exponent for the Cauchy problem \eqref{Semi_Generalized_General_Plate_Memory} with $u_1\in L^m(\mb{R}^n)$ carrying $m\in[1,2)$ and some $\sigma,n$ is given by $p_{\mathrm{crit}}(n,m,\sigma)=1+2m\sigma/n$ describing the threshold condition between global (in time) existence of small data weak solutions and blow-up of weak solutions even for small data.
\end{remark}
\begin{proof}
First of all, motivated by \cite{Dao-Reissig-2020}, we introduce a radial space-dependent test function $\varphi=\varphi(x)$ such that
\begin{align*}
\varphi(x):=\langle x\rangle^{-n-2s_{\sigma}}=(1+|x|^2)^{-n/2-s_{\sigma}},
\end{align*}
where $s_{\sigma}$ is chosen as an arbitrary constant belonging to $(0,1)$ if $\sigma$ is an integer number, and a small constant satisfying $0<s_{\sigma} \leqslant \sigma-[\sigma]$ if $\sigma$ is a fractional number. Then, Lemma \ref{lemma2.1} leads to the following estimate for any $\sigma\geqslant 1$:
\begin{equation}
\left|(-\Delta)^{\sigma} \langle x\rangle^{-n-2s_{\sigma}}\right| \lesssim \langle x\rangle^{-n-2s_{\sigma}}. \label{*}
\end{equation}
Moreover, we choose the time-dependent test function $\eta=\eta(t)$ such that $\eta\in\ml{C}_0^{\infty}([0,\infty))$ and
\begin{align}\label{Assum eta}
\eta(t):=\begin{cases}
1&\mbox{for} \ \ 0\leqslant t\leqslant\frac{1}{2},\\
\mbox{decreasing}&\mbox{for}\ \  \frac{1}{2}\leqslant t\leqslant 1,\\
0&\mbox{for} \ \ t\geqslant1,
\end{cases}
\end{align}
satisfying
\begin{align}\label{Assumption_eta}
(\eta(t))^{-\frac{p'}{p}}\left(|\eta'(t)|^{p'}+|\eta''(t)|^{p'}+|\eta'''(t)|^{p'}\right)\leqslant C\ \ \mbox{for any}\ \ t\in\left[\tfrac{1}{2},1\right],
\end{align}
where $C$ is a positive constant and $p'$ is the conjugate of $p$. Let $R$ be a large parameter in $[0,\infty)$. Then, we may introduce the test function
\begin{align*}
\psi_R(t,x):=\eta_R(t)\varphi_R(x):=\eta(t/R^{2\sigma})\varphi(x/R).
\end{align*}
To begin with the proof, we should define the functionals as follows:
\begin{align*}
I_R:=\int_0^{\infty}\int_{\mb{R}^n}|u(t,x)|^p\psi_R(t,x)\mathrm{d}x\mathrm{d}t\ \ \mbox{and}\ \ \widetilde{I}_R:=\int_0^{\infty}\int_{\mb{R}^n}|u(t,x)|^p\partial_t\psi_R(t,x)\mathrm{d}x\mathrm{d}t.
\end{align*}
Let us assume $u=u(t,x)$ being a global (in time) Sobolev solution taken from $\ml{C}\left([0,\infty),L^2(\mb{R}^n)\right)$ to \eqref{Semi_Generalized_Plate_Memory}. By performing once integration by parts, we obtain
\begin{align*}
I_R-\widetilde{I}_R&=\int_0^{\infty}\int_{\mb{R}^n}\left(u_{tt}(t,x)+(-\Delta)^{\sigma}u(t,x)+u(t,x)-(g\ast u)(t,x)\right)\psi_R(t,x)\mathrm{d}x\mathrm{d}t\\
&\quad+\int_0^{\infty}\int_{\mb{R}^n}\left(u_{ttt}(t,x)+(-\Delta)^{\sigma}u_t(t,x)+u_t(t,x)-u(t,x)+(g\ast u)(t,x)\right)\psi_R(t,x)\mathrm{d}x\mathrm{d}t\\
&\quad-\int_{\mb{R}^n}\left(u_{tt}(t,x)+(-\Delta)^{\sigma}u(t,x)+u(t,x)-(g\ast u)(t,x)\right)\psi_R(t,x)\big|_{t=0}^{t=\infty}\mathrm{d}x\\
&=\int_0^{\infty}\int_{\mb{R}^n}\left(u_{ttt}(t,x)+u_{tt}(t,x)+(-\Delta)^{\sigma}u_t(t,x)+(-\Delta)^{\sigma}u(t,x)+u_t(t,x)\right)\psi_R(t,x)\mathrm{d}x\mathrm{d}t,
\end{align*}
where we used the support condition for $\eta_R(t)$ and $u_{tt}(0,x)=0$ since $u(0,x)=0$.\\
Let us now apply several times integration by parts in the above identity to show
\begin{align*}
I_R-\widetilde{I}_R&=-\int_{\mb{R}^n}u_1(x)\varphi_R(x)\mathrm{d}x-\int_0^{\infty}\int_{\mb{R}^n}u(t,x)\left(\partial_t^3\psi_R(t,x)-\partial_t^2\psi_R(t,x)+\partial_t\psi_R(t,x)\right)\mathrm{d}x\mathrm{d}t\\
&\quad-\int_0^{\infty}\int_{\mb{R}^n}(-\Delta)^{\sigma}u(t,x)\left(\partial_t\psi_R(t,x)-\psi_R(t,x)\right)\mathrm{d}x\mathrm{d}t\\
&=:-\int_{\mb{R}^n}u_1(x)\varphi_R(x)\mathrm{d}x+J_{1,R}+J_{2,R}.
\end{align*}
To deal with the estimation of $|J_{1,R}|$, we employ H\"older's inequality 
\begin{align*}
|J_{1,R}|&\leqslant\int_{R^{2\sigma}/2}^{R^{2\sigma}}\int_{\mb{R}^n}|u(t,x)|\varphi_R(x)\left(|\eta_R'''(t)|+|\eta_R''(t)|+|\eta_R'(t)|\right)\mathrm{d}x\mathrm{d}t\\
&\leqslant I_R^{\frac{1}{p}}\left(\int_{R^{2\sigma}/2}^{R^{2\sigma}}\int_{\mb{R}^n}\varphi_R(x)(\eta_R(t))^{-\frac{p'}{p}}\left(|\eta_R'''(t)|^{p'}+|\eta_R''(t)|^{p'}+|\eta'_R(t)|^{p'}\right)\mathrm{d}x\mathrm{d}t\right)^{\frac{1}{p'}}\\
&\lesssim I_R^{\frac{1}{p}}\left(R^{-2\sigma p'+2\sigma+n}\int_{\mb{R}^n}\langle \tilde{x}\rangle^{-n-2s_{\sigma}}\mathrm{d}\tilde{x}\right)^{\frac{1}{p'}}\lesssim I_R^{\frac{1}{p}} R^{-2\sigma+\frac{2\sigma+n}{p'}},
\end{align*}
where we used the change of variables $\tilde{t}:=t/R^{2\sigma}$, $\tilde{x}:=x/R$ and our assumption \eqref{Assumption_eta}. Here, we should mention that the integral of $\langle \tilde{x}\rangle^{-n-2s_{\sigma}}$ over $\mb{R}^n$ is bounded due to $s_{\sigma} \in (0,1)$.

On the other hand, to estimate $|J_{2,R}|$, we notice that $\varphi_R\in H^{2\sigma}(\mb{R}^n)$ and $u\in\ml{C}\left([0,\infty),L^2(\mb{R}^n)\right)$. Then, the application of Lemma \ref{lemma2.3} implies
$$ \int_{\mb{R}^n}(-\Delta)^{\sigma}u(t,x) \varphi_R(x) \mathrm{d}x=\int_{\mb{R}^n}u(t,x) (-\Delta)^{\sigma}\varphi_R(x) \mathrm{d}x. $$
Thus, it follows immediately that
\begin{align*}
\int_{\mb{R}^n}(-\Delta)^{\sigma}u(t,x)\left(\partial_t\psi_R(t,x)-\psi_R(t,x)\right)\mathrm{d}x=\int_{\mb{R}^n}u(t,x)\left((-\Delta)^{\sigma}\partial_t\psi_R(t,x)-(-\Delta)^{\sigma}\psi_R(t,x)\right)\mathrm{d}x.
\end{align*}
For this reason, we arrive at
\begin{align*}
J_{2,R}&=\int_0^{\infty}\int_{\mb{R}^n}u(t,x)\left((-\Delta)^{\sigma}\psi_R(t,x)- (-\Delta)^{\sigma}\partial_t\psi_R(t,x)\right)\mathrm{d}x\mathrm{d}t\\
&=\int_0^{\infty}\int_{\mb{R}^n}u(t,x)(-\Delta)^{\sigma}\varphi_R(x)\left(\eta_R(t)-\eta'_R(t)\right)\mathrm{d}x\mathrm{d}t.
\end{align*}
The application of H\"older's inequality again leads to
\begin{align*}
|J_{2,R}|&\leqslant I_R^{\frac{1}{p}}\left(\int_0^{R^{2\sigma}}\int_{\mb{R}^n}\eta_R(t)(\varphi_R(x))^{-\frac{p'}{p}}|(-\Delta)^{\sigma}\varphi_R(x)|^{p'}\mathrm{d}x\mathrm{d}t\right.\\
&\qquad\quad\left.+\int_{R^{2\sigma}/2}^{R^{2\sigma}}\int_{\mb{R}^n}(\eta_R(t))^{-\frac{p'}{p}}|\eta'_R(t)|^{p'}(\varphi_R(x))^{-\frac{p'}{p}}|(-\Delta)^{\sigma}\varphi_R(x)|^{p'}\mathrm{d}x\mathrm{d}t\right)^{\frac{1}{p'}}\\
&\lesssim I_R^{\frac{1}{p}}\left(R^{-2\sigma p'+2\sigma+n}\int_{\mb{R}^n}(\varphi(\tilde{x}))^{-\frac{p'}{p}}|(-\Delta)^{\sigma}\varphi(\tilde{x})|^{p'}\mathrm{d}\tilde{x}\right)^{\frac{1}{p'}} \\
&\lesssim I_R^{\frac{1}{p}} \left(R^{-2\sigma p'+2\sigma+n}\int_{\mb{R}^n}\langle \tilde{x}\rangle^{-n-2s_{\sigma}}\mathrm{d}\tilde{x}\right)^{\frac{1}{p'}}\lesssim I_R^{\frac{1}{p}}R^{-2\sigma+\frac{2\sigma+n}{p'}},
\end{align*}
where we used the estimate \eqref{*} in the last line of the above chain inequality. Here, we also used the assumption \eqref{Assumption_eta} to estimate the second integral and the change of variables $\tilde{t}:=t/R^{2\sigma}$, $\tilde{x}:=x/R$. In addition, to get the second estimate in the previous chain estimation of $|J_{2,R}|$, the application of Lemma \ref{lemma2.2} gives the relation
\begin{align*}
(-\Delta)^{\sigma}\varphi_R(x)=R^{-2\sigma}(-\Delta)^{\sigma}\varphi(\tilde{x})
\end{align*}
for any $\sigma\geqslant 1$.

Collecting all derived estimates and applying Young's inequality, we may conclude
\begin{align*}
I_R-\widetilde{I}_R+\int_{\mb{R}^n}u_1(x)\varphi_R(x)\mathrm{d}x\leqslant C_0 I_R^{\frac{1}{p}}R^{-2\sigma+\frac{2\sigma+n}{p'}}\leqslant \frac{1}{p}I_R+\frac{C_0^{p'}}{p'}R^{-2\sigma p'+2\sigma+n}
\end{align*}
for some suitable constant $C_0>0$, in other words,
\begin{align}
\frac{1}{p'}I_R-\widetilde{I}_R+\int_{\mb{R}^n}u_1(x)\varphi_R(x)\mathrm{d}x\lesssim R^{-2\sigma p'+2\sigma+n}. \label{blow-up.1}
\end{align}
Due to the setting that the test function $\eta(t)$ is a non-increasing function, one has $-\eta'_R(t)\geqslant0$. In other words, it holds that $-\widetilde{I}_{R}\geqslant0$.

Next, we divide our discussion into two cases as follows: $m=1$ and $m\in(1,2)$.

\begin{description}
	\item \textbf{Case 1:} If $m=1$, the we have the condition $1<p<p_{\mathrm{crit}}(n,1,\sigma)$, which is equivalent to
	\begin{equation} \label{blow-up.2}
	-2\sigma p'+2\sigma+n<0.
	\end{equation}
	Since the assumption \eqref{u_1.Cond1} holds, there exists $R_0>0$ such that
	\begin{equation} \label{blow-up.3}
	\int_{\mb{R}^n}u_1(x)\varphi_R(x)\mathrm{d}x >0
	\end{equation}
	for all $R>R_0$. Hence, we may immediately obtain from \eqref{blow-up.1} that
	$$ \int_{\mb{R}^n}u_1(x)\varphi_R(x)\mathrm{d}x \lesssim R^{-2\sigma p'+2\sigma+n} \to 0 \ \  \text{as}\ \ R\to \infty $$
	due to \eqref{blow-up.2}. It contracts to \eqref{blow-up.3}. Therefore, every global (in time) Sobolev solution blows up.
	\item \textbf{Case 2:} If $m\in(1,2)$, then we have the condition $1<p<p_{\mathrm{crit}}(n,m,\sigma)$, which is equivalent to
	\begin{equation} \label{blow-up.4}
	n\left(1-\tfrac{1}{m}\right)>-2\sigma p'+2\sigma+n.
	\end{equation}
	The assumption \eqref{u_1.Cond2} shows that
	\begin{align*}
	\int_{\mb{R}^n}u_1(x)\varphi_R(x)\mathrm{d}x&\geqslant\int_{|x|\leqslant R}u_1(x)\langle x/R\rangle^{-(n+2s_\sigma)}\mathrm{d}x\\
	&\gtrsim\int_{|x|\leqslant R}u_1(x)\mathrm{d}x\geqslant\int_{|x|\leqslant R}|x|^{-\frac{n}{m}}(\log(|x|))^{-1}\mathrm{d}x\\
	&\gtrsim(\log(R))^{-1}R^{n(1-\frac{1}{m})}
	\end{align*}
	for sufficiently large $R$, where we used a change of variable in the last estimate. In other words, we may claim from \eqref{blow-up.1} that
	\begin{align*}
	(\log(R))^{-1}R^{n(1-\frac{1}{m})}\lesssim\int_{\mb{R}^n}u_1(x)\varphi_R(x)\mathrm{d}x\lesssim R^{-2\sigma p'+2\sigma+n},
	\end{align*}
	that is,
	\begin{align*}
	(\log(R))^{-1}R^{n(1-\frac{1}{m})-(-2\sigma p'+2\sigma+n)}\lesssim 1
	\end{align*}
	for sufficiently large $R$. Letting $R\to\infty$, one may directly obtain a contradiction due to \eqref{blow-up.4}. All in all, every global (in time) Sobolev solution blows up.
\end{description}

Summarizing, the proof of Theorem \ref{Blow-up.Thoerem} is complete.
\end{proof}

\section{Final remarks}\label{Section_Finalremark}
\setcounter{equation}{0}
In Section \ref{Section_GESDS}, we have proved the global (in time) existence of small data Sobolev solutions with lower regularity to the Cauchy problem \eqref{Semi_Generalized_Plate_Memory}. One may also prove the global (in time) existence of small data energy solutions or higher-order energy solution such that
\begin{align*}
u\in\ml{C}\left([0,\infty),H^{\sigma+s}(\mb{R}^n)\right)\cap \ml{C}^1\left([0,\infty),H^{s+\ell}(\mb{R}^n)\right)
\end{align*}
with some suitable choice of $\ell$ which represents the regularity-loss-type, where $s\geqslant0$. The main approach is based on Theorem \ref{Thm_Est_Solution_Derivative} and some tools in Harmonic Analysis. It would be interesting to analyze the suitable relation between $s,\sigma$ and $\ell$ from regularity of initial data due to the regularity-loss-type decay property.

Throughout this paper, we have investigated global (in time) existence of small data Sobolev solutions and blow-up result of the obtained global solutions even for small data to the semilinear $\sigma$-evolution equations with exponential decay memory term. Moreover, we have determined the critical exponent $p=p_{\mathrm{crit}}(n,m,\sigma)$ for \eqref{Semi_Generalized_Plate_Memory} with some parameters $n,\sigma$ and $m$. However, it is still open to find the critical exponent for the semilinear $\sigma$-evolution models with general exponential decay memory kernel, namely,
\begin{align}\label{Semi_Generalized_General_Plate_Memory}
\begin{cases}
u_{tt}+(-\Delta)^{\sigma}u+u-g\ast u=|u|^p,&x\in\mb{R}^n,\ t>0,\\
u(0,x)=u_0(x),\ u_t(0,x)=u_1(x),&x\in\mb{R}^n,
\end{cases}
\end{align}
where $\sigma\geqslant 1$, and the memory kernel $g:[0,\infty)\to[0,\infty)$ denotes a time-dependent function having small perturbation of exponential decay such that
\begin{align}\label{General_Memory}
0<g(0)\,\mathrm{e}^{-C_1t}\leqslant g(t)\leqslant g(0)\,\mathrm{e}^{-C_2t}
\end{align}
with positive constants $C_1$ and $C_2$ for any $t\geqslant0$. Actually, the corresponding linear Cauchy problem for \eqref{Semi_Generalized_General_Plate_Memory} with $\sigma=2$ and \eqref{General_Memory} has been studied in \cite{Liu-Ueda-2020} recently. Therefore, it is reasonable to consider the memory kernel function satisfying \eqref{General_Memory}. The main difficulties to treat \eqref{Semi_Generalized_General_Plate_Memory} are not only non-local operator $(-\Delta)^{\sigma}$ when $\sigma$ is a fractional number but also to understand the treatment of flexible function $g(t)$. Due to the fact that $g(t)$ still has the exponential decay property, we may conjecture the critical exponent for \eqref{Semi_Generalized_General_Plate_Memory} is still given by
\begin{align*}
p=p_{\mathrm{crit}}(n,m,\sigma)= 1+\frac{2m\sigma}{n},
\end{align*} 
where initial data is taken from $L^m$ spaces with $m\in[1,2)$.

\appendix
\section{Tools from Harmonic Analysis}\label{Harmonic_Anal}
\setcounter{equation}{0}
\begin{prop}[Fractional Gagliardo-Nirenberg Inequality]\label{fractionalgagliardonirenbergineq}
	Let $p,\,p_0,\,p_1\in(1,\infty)$ and $\kappa\in[0,s)$ with $s>0$. Then, it holds for all $f\in L^{p_0}(\mb{R}^n)\cap \dot{H}^{s}_{p_1}(\mb{R}^n)$
	\begin{equation*}
	\|f\|_{\dot{H}^{\kappa}_{p}(\mb{R}^n)}\lesssim\|f\|_{L^{p_0}(\mb{R}^n)}^{1-\beta}\,\|f\|^{\beta}_{\dot{H}^{s}_{p_1}(\mb{R}^n)},
	\end{equation*}
	where $\beta=\beta_{\kappa,s}(p,p_0,p_1,n)=\big(\frac{1}{p_0}-\frac{1}{p}+\frac{\kappa}{n}\big)\big\backslash\big(\frac{1}{p_0}-\frac{1}{p_1}+\frac{s}{n}\big)$ and $\beta\in[\kappa/s,1]$.
\end{prop}
The proof of this result may be found in \cite{Hajaiej2011}.
\begin{prop}[Fractional Leibniz Rule]\label{fractionleibnizrule} 
	Let $s>0$, $1\leqslant  r\leqslant \infty$ and $1<p_1,\,p_2,\,q_1,\,q_2\leqslant \infty$ satisfying
	\begin{equation*}
	\frac{1}{r}=\frac{1}{p_1}+\frac{1}{p_2}=\frac{1}{q_1}+\frac{1}{q_2}.
	\end{equation*}
	Then, it holds for $f\in\dot{H}^{s}_{p_1}(\mb{R}^n)\cap L^{q_1}(\mb{R}^n)$ and $g\in\dot{H}^{s}_{q_2}(\mb{R}^n)\cap L^{q_2}(\mb{R}^n)$
	\begin{equation*}
	\|fg\|_{\dot{H}^{s}_{r}(\mb{R}^n)}\lesssim \|f\|_{\dot{H}^{s}_{p_1}(\mb{R}^n)}\,\|g\|_{L^{p_2}(\mb{R}^n)}+\|f\|_{L^{q_1}(\mb{R}^n)}\,\|g\|_{\dot{H}^{s}_{q_2}(\mb{R}^n)}.
	\end{equation*}
\end{prop}
The proof of the above inequality can be found in \cite{Grafakos2014}.
\begin{prop}[Fractional Chain Rule]\label{fractionalchainrule}
	Let $s>0$, $p>\lceil s\rceil$ and $1<r,\,r_1,\,r_2<\infty$ satisfying
	\begin{equation*}
	\frac{1}{r}=\frac{p-1}{r_1}+\frac{1}{r_2}.
	\end{equation*}
	Then, it holds for $f\in\dot{H}^{s}_{r_2}(\mb{R}^n)\cap L^{r_1}(\mb{R}^n)$
	\begin{equation*}
	\|\pm f|f|^{p-1}\|_{\dot{H}^{s}_{r}(\mb{R}^n)}+\|\,|f|^p\|_{\dot{H}^{s}_{r}(\mb{R}^n)}\lesssim\|f\|_{L^{r_1}(\mb{R}^n)}^{p-1}\,\|f\|_{\dot{H}^{s}_{r_2}(\mb{R}^n)}.
	\end{equation*}
\end{prop}
We can find the proof of this proposition in \cite{PalmieriReissig2018}.

\section*{Acknowledgments}
The Ph.D. study of Wenhui Chen are supported by S\"achsisches Landesgraduiertenstipendium. This work was partially written while Wenhui Chen was a Ph.D. student at TU Freiberg.

\end{document}